\documentclass[12pt, reqno]{amsart}
\usepackage{amsmath}

\newtheorem{theor}{\indent\sc Theorem}[section]
\newtheorem{corol}[theor]{\indent\sc Corollary}
\newtheorem{lemma}[theor]{\indent\sc Lemma}   
\newtheorem{prop}[theor]{\indent\sc Proposition}

\newtheorem{dfn}[theor]{\indent\sc Definition}

\newtheorem{example}[theor]{\indent\sc Example}

\newcommand{\reals }{{\mathbb R}}
\newcommand{\complex }{{\mathbb C}}
\newcommand{\lorentz }{{\mathbb L}}

\newcommand{\mink}{{\mathbb R}_{1}^{4}}

\newcommand{\hyper}{{\mathbb H}^{3}}

\newcommand{\lpr}[2]{{{\langle}{#1},{#2}{\rangle}}} 

\newcommand{\nlight}[1]{(1+{#1}{\overline{{#1}}},{#1}+{\overline{{#1}}},
-i({#1}-{\overline{{#1}}}),-1+{#1}{\overline{{#1}}})} 

\newcommand{\slight}[1]{(1+{#1}{\overline{{#1}}},{#1}+{\overline{{#1}}},
-i({#1}-{\overline{{#1}}}),1-{#1}{\overline{{#1}}})}

\newcommand{\weis}[2]{({#1}+{#2},1+{#1}{#2},i(1-{#1}{#2}),{#1}-{#2})}

\newcommand {\C}{{\mathcal C}}

\newcommand {\X}{{\mathfrak X}}

\newcommand {\Log}{\mbox{Log}\;}

\begin{document}
 \baselineskip=1.0\baselineskip
  
\title[]{Some Partial Differential Equations and conformal surfaces of the 4-dimensional Minkowski space}
\author[]{M.P. Dussan{$^1$}, \ \  A.P. Franco Filho, \ \ P. Sim\~oes}  

\footnote{Corresponding author.

 E-mail addresses: dussan@ime.usp.br (M.P. Dussan), apadua@ime.usp.br (A.P. Franco Filho)}
\maketitle

\centerline{\small{ Universidade de S\~ao Paulo, Departamento de
Matem\'atica - IME}}
\centerline{\small{CEP: 05508-090. S\~ao Paulo. Brazil}}

\begin{abstract}
This paper introduces a complex representation for spacelike surfaces in the Lorentz-Minkowski space $\mathbb L^4$, 
based in two complex valued functions which can be assumed to be holomorphic or anti-holomorphic. 
When the immersion is contained in quadrics of $\mathbb L^4$, the representation then allows us to obtain 
interesting partial differential equations with holomorphic or anti-holomorphic parameters, within which we find the partial 
Riccati Equation. Using then theory of holomorphic complex functions we construct explicitly new local solutions for those PDEs 
together with its associated geometric solutions. So, several explicit examples are given. 
As geometric consequence, through of our approach we characterize all conformal totally umbilical spacelike immersions into $\mathbb L^4$, and in addition, 
we also show that for each conformal immersion in $\mathbb L^4$ which satisfies the partial Riccati equation there exists a 
Bryant immersion in $\mathbb H^3$, both immersions being congruent by a translation vector. 
\end{abstract}

\vspace{0.2cm}

\vspace{0.05cm}

\section{Introduction}
There are several results involving the theory of spacelike surfaces in the Lorentz-Minkowski space $\mathbb L^4$ and in particular, 
those involving the geometry of spacelike surfaces through of the  future light cone, which is very rich and important. 
The results found in the literature came through from different viewpoints, see \cite{L}, \cite{LJ}, \cite{LUY}, \cite{PRR}, \cite{PR}, 
for instance. Some viewpoints involve  basic formulations to obtain the first non-trivial eigenvalue for the Laplacian 
or to obtain special lightlike normal vector fields, which allow to prove results about the local geometry of the surfaces. 
Our approach is to consider PDEs in complex variables, as well as  analysis of holomorphic or anti-holomorphic maps.

\vspace{0.1cm}
 
In fact,  this paper is  focused on spacelike surfaces in $\mathbb L^4$ with certain extrinsic geometric 
conditions via PDEs which appear when we represent  the surfaces locally using a special complex representation for the 
complex derivative of the immersion. Our complex representation involves two 
functions $a$ and $b$ which may be taken to be holomorphic or anti-holomorphic, depending in which quadric is made 
the immersion and which are the geometric conditions assumed.  We note that in a simultaneous paper of the two first authors 
together M. Magid (\cite{DPM}), our complex representation formula obtained here was  already adapted for a split-complex 
representation (or paracomplex) of the derivative, with the objective of obtaining a Weierstrass representation of 
timelike surfaces in $\mathbb R^4_1$  that were solutions of the Bjorling problem in $\mathbb R^4_1$.  
A priori, our complex representation for spacelike surfaces in $\mathbb R^4_1$ established here, it was also already used by 
A. Asperti and J. Vilhena in \cite{AV} to solve the Bjorling problem for maximal spacelike  surfaces in  $\mathbb R^4_1$. 

\vspace{0.1cm}
Inspired by the Weierstrass representation obtained by Hoffman and Osserman in \cite{HO1} and \cite{HO2}, we develop our complex representation assuming two 
smooth maps $a, \ b$ to be complex valued and we look for conditions 
for the existence of a conformal spacelike immersion $(M,f)$  in $\mathbb L^4$ such that $f_w (w) = \mu (w) W(a(w), b(w))$ where 
$\mu$ is a complex integrability factor and the complex valued function $W$ is conveniently given in terms of $a$ and $b$. 
Then assuming $a$ and $b$ to be holomorphic or anti-holomorphic maps, we find interesting partial differential equations whose local 
solutions will be associated to conformal spacelike surfaces in $\mathbb L^4$ or in the lightlike cone with certain curvature 
conditions, as well as, to be associated to conformal spacelike surfaces in the hyperbolic space with lightlike mean curvature. 
Using \cite{H} and \cite{LR}, we also focus on those equations and to construct, in a systematic way, local solutions for those PDEs together with its associated  
explicit spacelike surfaces. For instance, for conformal, totally umbilical spacelike surface in $\mathbb L^4$, we study certain 
PDEs involving two anti-holomorphic parameters, which is associated to conformal immersions into the light cone. 
In fact, analyzing that PDE, we establish  conditions for obtaining positive real valued solutions and as immediate consequence, 
we show that the spacelike surface has to be congruent to a spherical surface $S^2(1/r) \subset \mathbb R^3$ or to  an hyperbolic 
surface $S^2(-1/r) \subset \mathbb R^3_1$.  

In another direction, we also find the partial Riccati equation with two holomorphic parameters, which appears when describing spacelike conformal surfaces in hyperbolic space $\mathbb H^3$ with lightlike mean curvature vector field. In general, these kind of surfaces are well known as Bryant surfaces (\cite{B}). In fact,
 through of several computations we first find a PDE equation linking the functions $a$, $b$ and under a variable change we then find the partial Riccati Equation for which we study local solutions. We construct new local solutions from two given solutions and so, in addition, we identify the geometry associated to the solution. Note that it is difficult to find explicit solutions for the Riccati Equation, so in this paper we give a systematic tool which will allow to go ahead when finding its solutions. \ Then several explicit examples of solutions of partial Riccati Equation and its associate geometry are given. In addition, we also prove one of our main results, which establishes that for each conformal spacelike  surface $M$ in $\mathbb L^4$, which satisfies the Riccati equation with two holomorphic parameters, there exists a Bryant immersion of $M$  in $\mathbb H^3$, such that both surfaces
are congruent each other by a translation vector.

\vspace{0.2cm}
 Section 2 of this paper, is devoted to preliminaries where we establish notation and the importance of maps 
 $a$ and $b$.  We also establish the theory about integrability of complex 1-forms, 
 which will be used to obtain a Weierstrass representation type. 
 Section 3  develops theory about conformal surfaces in the Minkowski-Lorentz  Space $\mathbb L^4$.  In this section we discover  the first partial differential equations involving the maps $a, b$ and $\mu$, namely, equations corresponding to compatibility condition equations and to the factor of integration $\mu$.  Theorem  \ref{110} shows the Weierstrass representation for the surfaces. Section 4 is devoted to conformal immersions into the future
  directed light cone $\mathcal C$. Here, it is found the partial differential equation (\ref{17}), this relates to
 the existence of conformal immersions in the light cone such that  $f_w = \mu(w) W(a, b)$. Some explicit solutions are also found. Theorem 
 \ref{1101} establishes  
 that,  for every conformal totally umbilical immersion from $M$ into $\mathbb L^4$, with $a$ and $b$ anti-holomorphic, there exists a conformal immersion in the 
 light cone, which is similar up to translation. Moreover, the totally umbilical surface has to be congruent to a spherical surface $S^2(1/r) \subset \mathbb R^3$ or to  an hyperbolic surface $S^2(-1/r) \subset \mathbb R^3_1$. Section 5 focuses on the study of partial differential equations associated to the case when the mean curvature vector
 is lightlike in a immersion in the 3-dimensional hyperbolic space. We discover the partial Riccati Equation (\ref{111})
  with two holomorphic parameters when finding the linking equation
 for the maps $a$ and $b$, where $a$ is holomorphic and $b$ is neither holomorphic or anti-holomorphic. Section 6 is devoted to study
 geometric solutions associated to the  partial Riccati equation with holomorphic parameters.
 Then Lemma \ref{54},  Proposition \ref{46.1} together  Proposition \ref{392} give then a systematic tool to construct explicit solutions for the Riccati Equation and  Bryant surfaces.  Several explicit examples are then given. Section 6 finishes  with Theorem \ref{96} which establishes  that every conformal immersion
$(M,f)$ in $\mathbb L^4$ which satisfies the partial Riccati Equation, has to be congruent to a Bryant immersion $(M,h)$ in $\mathbb H^3$. 
 
\section{Preliminaries} 

The Minkowski vector space $\mathbb L^4= \mink$ is the real vector space $\reals^{4}$ endowed with the usual Euclidean topology, 
with the semi-Riemannian metric 
$$\langle \;, \; \rangle = -(d x_{0})^{2} + (d x_{1})^{2} + (d x_{2})^{2} + (d x_{3})^{2},$$ 
oriented vectorially by $\partial_{0} \wedge \partial_{1} \wedge\partial_{2} \wedge\partial_{3}$ and temporally by $\partial_{0}$, where  $\{\partial_{0}, \partial_{1}, \partial_{2}, \partial_{3}\}$ is 
the canonical basis of $\mathbb L^4$.

On $\mathbb L^4$ let $\reals^{3}$ be the vector subspace generated by $\partial_{1}, \partial_{2}, \partial_{3}$. The vectors of 
$\reals^{3}$ are indicated by an upper hat. Namely, if $v \in \reals^{3}$, it is indicated by $\hat{v}$. 

An orthonormal basis $\{E_{k}\}$ of $\mathbb L^4$ is a Minkowski referential if $E_{0}$ is a timelike future directed and 
$E_{0} \wedge E_{1} \wedge E_{2} \wedge E_{3} = \partial_{0} \wedge \partial_{1} \wedge \partial_{2} \wedge \partial_{3}$.

Given $u, v, w \in \mathbb L^4$ the vector ${\X}(u,v,w)$ defined by
$$\lpr{\X(u,v,w)}{z} = - dx_{0} \wedge dx_{1} \wedge dx_{2} \wedge dx_{3} (z,u,v,w), \ \  \; \; \forall z \in \mathbb L^4$$ 
is the cross product of $(u,v,w)$. 

In particular $\X(\partial_{0},u,v) = \hat{u} \times \hat{v}$ is the usual cross product of $\reals^{3}$.

\begin{prop}
Let $V$ be a spacelike vector subspace of $\mathbb L^4$. Then $\{ u,v \} \subset V$ is a linearly dependent set on $\mathbb L^4$ if and only if 
$\{ \hat{u}, \hat{v} \}$ is a linearly dependent set on $\reals^{3}$.
\end{prop}
\begin{proof}
In fact, if $\hat{v} = \lambda \hat{u}$ then $v - \lambda u = (v_{0} - \lambda u_{0}) \partial_{0}$. Since $V$ is a 
spacelike vector subspace of $\mathbb L^4$, it follows from 
$\lpr{(v_{0} - \lambda u_{0}) \partial_{0}}{(v_{0} - \lambda u_{0}) \partial_{0}} = - (v_{0} - \lambda u_{0})^{2} \geq 0,$ 
that $v_{0} = \lambda u_{0}$. Therefore $v = \lambda u$.
\end{proof}

If the restriction of $\lpr{\;}{\;}$ to a hyperplane, that is a 3-dimensional vector subspace, is positive definite, indefinite 
or degenerated, then we have respectively spacelike, timelike or lightlike hyperplanes.

In what follows, a plane will be, unless we explicitly say otherwise, a spacelike 2-dimensional vector subspace of 
$\lorentz^{4}$. If $r$ is a positive real number, a $r$-isothermal basis of a plane $E^{2}$ is an ordered basis 
$(X,Y)$ such that 
$$\lpr{X}{X} = \lpr{Y}{Y} = r^{2} \; \; \mbox{ and } \; \; \lpr{X}{Y} = 0.$$ 
On such cases we consider the plane $E^{2}$ oriented by $(X,Y)$. 

The complexification of $\lpr{\;}{\;}$ gives rise to a symmetric bilinear form on $\complex^{4}$, 
also denoted by $\lpr{\;}{\;}$. 
Defining the complex vector $Z = X + i Y \in \complex^{4}$, such that $\lpr{Z}{Z} = 0$ and 
$\lpr{Z}{\overline{Z}} = 2 r^{2} > 0$, 
we have a plane $E^{2} = [X,Y] = \{ \alpha X + \beta Y \in \lorentz^{4} | \ \ \alpha,\beta \in \reals \}$. 
We also note that
$\lpr{Z}{Z} = 0$ and $\lpr{Z}{\overline{Z}} = 2 r^{2} > 0$ imply that $\lpr{X}{X} = \lpr{Y}{Y} = r^{2}$ and $\lpr{X}{Y} = 0$. 

Then it follows quickly that:

\begin{prop}\label{199}
\begin{enumerate} 
\item  Let  $Z=(Z_{0}, Z_{1}, Z_{2},Z_{3}) \in \complex^{4}$  be such that  
$\lpr{Z}{Z} = 0, $
\newline $ \lpr{Z}{\overline{Z}} = 2 r^{2} > 0$ and $ Z_{1} - i Z_{2} \neq 0.$ 
It defines the functions $\mu, \ a $ and \ $b$ \ by 
$$ \mu = \frac{Z_{1} - i Z_{2}}{2}, \; \; \ \ \   a = \frac{Z_{0} + Z_{3}}{Z_{1} -i Z_{2}},  \; \; \ \ \ \  b = \frac{Z_{0} - Z_{3}}{Z_{1} - i Z_{2}}.$$ 
Then it follows that \ $a \overline{b} - 1 \neq 0$. Moreover 
$$
Z = \mu W(a,b) \; \; \; {\rm and} \; \; \; \;
 W(a,b) = \weis{a}{b}.
$$

\item \ Let  $E^{2} = [X,Y]$ be  a fixed spacelike plane. Then its orthogonal complement relative to $\lpr{\;}{\;}$  is the timelike two dimensional 
vector subspace $(E^{2})^{\perp} = [L_{0}(b),L_{3}(a)]$ where $L_{0}(b),L_{3}(a)$ are future directed lightlike vectors given by 
$$
\begin{cases} 
L_{0}(b) = \slight{b} \\ 
 L_{3}(a) = \nlight{a}.
 \end{cases}
 $$ 
\end{enumerate}
\end{prop}

\vspace{0.3cm}
We note in particular that $\lpr{L_{0}(b)}{L_{3}(a)} = - 2(1 - a \overline{b})(1 - b \overline{a}).$ 

\vspace{0.3cm}

Next in this paper, we look for answers for the following question: 

Let $a,b : M \to \mathbb C$ be two smooth maps and $W(a(w),b(w))$ the family of 2-dimensional spacelike vector 
subspaces of $\mathbb L^4$, 
 where 
$$W(a,b) = (a + b,1 + ab,i (1-ab),a- b).$$  Under which conditions we can guarantee the existence of a conformal spacelike immersion $ (M,f)$ from $M$ into $\mathbb L^4$, such that $f_w(w) = \mu(w) W(a(w),b(w))$?
In fact, we will show in next sections, that the existence depends of the solution of the PDE system involving the maps $a$ and $b$. When assuming that $a$ or $b$ are holomorphic or anti-holomorphic, we also show in particular, that the existence of the immersion depends from the existence of solutions of certain partial equations, within which we find the partial Riccati Equation.

 \vspace{0.3cm}

Throughout this paper, $M$ will be an open connected and simply connected subset of the set of the 
complex numbers $\complex$. We will denote by $\mathcal{H}(M)$ the set of holomorphic maps from $M \subset \complex$ 
into $\complex$. A map $f = P + i Q$ from $M$ into $\complex$ is an anti-holomorphic map if, and only if, its conjugate map
$\overline{f} = P - iQ$ is a holomorphic map. The set of all anti-holomorphic maps will be denoted by 
$\overline{\mathcal{H}}(M)$. The set of all continuously differentiable maps from $M$ into $\complex$ we will be denoted by 
${C}^{\infty}(M,\complex)$, and we say that these maps are smooth maps from $M$ into $\complex$.

Let
$$\frac{\partial}{\partial w} = \frac{1}{2} \left( \frac{\partial}{\partial u} - i \frac{\partial}{\partial v} \right) 
\; \; \ \ \mbox{ and } \; \; \ \ 
\frac{\partial}{\partial \overline{w}} = \frac{1}{2} \left( \frac{\partial}{\partial u} + i \frac{\partial}{\partial v} 
\right)$$
be the differential operators defined over the set of all smooth maps from $M$ into $\complex$, where $w = u + i v \in M$. 

It follows that 
a smooth map $f$ from $M$ into $\complex$ 
is a holomorphic map if and only if $\frac{\partial}{\partial \overline{w}} f(w) = 0$ for all $ w \in M$. 

Here we  will also use often  
the notation $\frac{\partial f}{\partial w} = f_{w}$ \ and \ $\frac{\partial f}{\partial \overline{w}} = f_{\overline{w}}$. 

\vspace{0.3cm}

Next we develop the theory about integrability of complex 1-forms. 

\vspace{0.2cm}

Let $\beta = \phi dw$ be a complex $1$-form where $\phi = P - iQ$. So,  
$$\beta = (P - iQ)(du + i dv) = (P du + Q dv) + i(P dv - Q du).$$ 
We will say that the complex $1$-form $\beta$ has no real periods if, and only if, the real $1$-form 
$$\Re(\beta) = P du + Q dv$$ 
is a closed $1$-form, i.e., $d(\Re(\beta)) = (-P_{v} + Q_{u})du \wedge dv = 0$, which is equivalent to $P_{v} = Q_{u}.$ 

In this case, the line integral 
$$I(\gamma,w) = \int_{\gamma} \Re(\phi(\gamma(t)) \gamma'(t)) dt$$ 
does not depends of the continuous path of integration $\gamma : [0,1] \longrightarrow M$ where $\gamma(0) = w_{0}$, 
 $\gamma(1) = w$
and such that there is $\gamma'(t)$ except for a finite set of points.    

\vspace{0.2cm}

Then we take by definition $\int_{w_{0}}^{w} \Re (\beta) := \int_{\gamma} \Re(\phi(\gamma(t)) \gamma'(t)) dt.$

\vspace{0.2cm}
Now, let 
$$L(\gamma,w) = \int_{\gamma} \phi(\gamma(t)) \gamma'(t) dt = I(\gamma,w) + i \int_{\gamma} \Im(\phi(\gamma(t)) \gamma'(t)) dt$$ 
be the integration of the complex $1$-form $\beta = \phi dw$, over a good path $\gamma$. 

Since, $I(\gamma,w)$ only depends of the initial point $w_{0}$ and of the final point $w$, we have that 
$$\Re(L(\gamma,w)) = I(\gamma,w) = \int_{w_{0}}^{w} \Re (\beta)$$ 
and then we define 
$$\Re \int_{w_{0}}^{w} \beta = \Re \int_{\gamma} \phi(\gamma(t)) \gamma'(t) dt = I(\gamma,w) = \int_{w_{0}}^{w} \Re (\beta).$$

\vspace{0.2cm}

\begin{prop}\label{119}
A complex $1$-form $\beta = \phi dw$ has no Real periods if, and only if, $\phi_{\overline{w}}$ is  real. Then, in this case, for each $w \in M$:  
$$\frac{\partial}{\partial w} \left(2 \Re \int_{w_{0}}^{w} \beta \right) = \phi(w).$$
\end{prop}
\begin{proof}
Define the real valued function $$\Psi = 2 \Re \int_{w_{0}}^{w} \beta = 2 \int_{w_{0}}^{w} P du + Q dv.$$ 
Then, it follows from Calculus that $\nabla \Psi(u,v) = (\Psi_{u},\Psi_{v}) = (2 P(u,v), 2Q(u,v))$, so 
$$\Psi_w = \frac{1}{2}\left(\frac{\partial \Psi}{\partial u} - i \frac{\partial \Psi}{\partial v}\right) = P(u,v) - i Q(u,v) = \phi(w).$$
\end{proof}

\section{Conformal Surfaces in the Minkowski Space $\lorentz^{4}$}

\begin{dfn}
A conformal surface of the Minkowski space $\lorentz^{4}$ is a pair $(M,f)$, where $M \subset \complex$ is a 
connected simply connected open subspace of the complex plane and $f:M \longrightarrow \lorentz^{4}$ is an immersion 
such that, for each $w \in M$, 
\begin{equation}
\lpr{f_{w}(w)}{f_{w}(w)} =0 \ \ \; \mbox{and}  \ \ \; \lpr{f_{w}(w)}{{\overline{f_{w}(w)}}} = \frac{r^{2}(w)}{2}, \; \ \ 
\mbox{ where } \; r(w) > 0.
\end{equation}
\end{dfn}

We observe that the second condition in (1) implies that $\lpr{f_u}{f_u} = \lpr{f_v}{f_v} = r^2$, and  
hence the metric tensor of $(M,f)$ is such that  $g_{11} = g_{22} = r^2$ and $g_{12} =0$. This means the metric tensor is given by 
$d s^{2}(f) = r^{2} d w d{\overline{w}}.$   

\vspace{0.3cm}

In this paper we are interested only in surfaces for which 
$$ {f_{1}}_{w}(w) - i {f_{2}}_{w}(w) \neq 0, \ \ \forall w \in M.
$$
where $f_j, \ j=0, 1,2,3$ denotes the components of $f$. 
Hence $f_{w}(w) = \mu(w) W(a(w),b(w)),$ and we set 
\begin{equation}\label{1'}
a(w) = \frac{{f_0}_{w}(w) + {f_3}_{w}(w)}{{f_1}_{w}(w) - i {f_2}_{w}(w)}, \; \;
b(w) = \frac{{f_0}{w}(w) - {f_3}_{w}(w)}{{f_1}_{w}(w) - i {f_2}_{w}(w)}, \; \;
\mu(w) = \frac{{f_1}_{w}(w) - i {f_2}_{w}(w)}{2},
\end{equation}

where we also assume that both maps $a(w)$ and $b(w)$ are not constant maps. 

\vspace{0.3cm}

Then we have the following proposition.

\begin{prop} Let $(M,f)$ be a conformal immersion from $M$ into $\lorentz^{4}$. 
Then the first and second fundamental quadratic form of the immersion are respectively
$$ds^{2}(f) = 2 \lpr{f_{w}}{f_{\overline{w}}} dw d{\overline{w}} = 4 
\mu {\overline{\mu}} (1 - a \overline{b})(1 - b \overline{a}) dw d{\overline{w}},$$
$$B(f) = [f_{w w}]^{\perp} dw dw + 2 f_{w \overline{w}} dw d{\overline{w}} + 
\overline{[f_{w w}]^{\perp}} d{\overline{w}} d{\overline{w}},$$ 

and the shape operator is given by
$$\lpr{A_{f}^{(\vec{n})}(X)}{Y} = \lpr{B(f)(X,Y)}{\vec{n}} \; \; \mbox{ for each } \; \; \vec{n} \in [f_{w}]^{\perp} \; \; with \ \ 
X,Y \in [f_{w}].$$
\end{prop}

\begin{prop}\label{0.1}
Let $(M,f)$ be a conformal immersion from $M$ into $\lorentz^{4}$, and $\vec{v} = (v_{0},v_{1},v_{2},v_{3})$  be a vector in  $\lorentz^{4}$.
\begin{enumerate}
\item \ The surface $S = f(M)$ is contained into the affine hyperplane 
$H_{\vec{v}} = p_{0} + [\vec{v}]^{\perp} = \{p \in \lorentz^{4} : \lpr{p - p_{0}}{\vec{v}} = 0 \}$ if, and only if, for all $w \in M$, 
$\lpr{f_{w}}{\vec{v}} = 0$.
\vspace{0.2cm}

\item \ If  $\vec{v}$ is such that $\lpr{\vec{v}}{\vec{v}} \neq 0$ and unitary, then
 $S \subset H_{\vec{v}}$ if and only if  there exists a M$\ddot{o}$bius 
transformation $M_{\vec{v}} : \complex \longrightarrow \complex$ defined by 
\begin{equation*}
M_{\vec{v}}(p) = \frac{V_{0} \; p - Z}{V_{3} + \overline{Z}\; p}, \ \; \mbox{with } \; \ V_{0} = v_{0} + v_{3}, \; \ 
V_{3} = - v_{0} + v_{3}, \;  \ Z = v_{1} + i v_{2},
\end{equation*}
such that for each $w \in M$, \ $a(w) = M(b(w))$.
\end{enumerate}
\end{prop}

\begin{proof} First it notes that 
$\lpr{f_{w}}{\vec{v}} = 0$ if and only if
$-(a + b)v_{0} + (a - b)v_{3} + (1 + ab)v_{1}  + i(1 - ab)v_{2} =  
a V_{3} - bV_{0} + Z + ab{\overline{Z}} = 0$. Now, if \ $b(w_{0}) = \frac{-V_{3}}{\overline{Z}}$ \ 
then 
$$0 = a V_{3} - \frac{-V_{3}}{\overline{Z}}V_{0} + Z + a\frac{-V_{3}}{\overline{Z}} {\overline{Z}} = 
\frac{V_{0} V_{3} + Z \overline{Z}}{\overline{Z}} \ \ if \ and \ only \ if \  \; \lpr{\vec{v}}{\vec{v}} = 0.$$ 
Finally, if  we assume that $\lpr{\vec{v}}{\vec{v}} = \pm 1$ it follows that $a(w) = \frac{V_{0} \; b(w) - Z}{V_{3} + \overline{Z}\; b(w)}.$
\end{proof}
\begin{corol}
$S = f(M)$ is contained into an affine lightlike hyperplane if, and only if, either $a(f)$ or $b(f)$ is a 
constant map.

\end{corol}
\begin{proof}
In fact in this case we have that $\lpr{\vec{v}}{\vec{v}} = 0$ and $\lpr{f_{w}}{\vec{v}} = 0$. \ If $a(w) = \frac{V_{0} \; b(w) - Z}{V_{3} + \overline{Z}\; b(w)}$, then, 
$a_{w} = \frac{(V_{0}V_{3} + Z \overline{Z}) b_{w}}{(V_{3} + \overline{Z}\; b(w))^{2}} = 0$ and 
$a_{\overline{w}} = \frac{(V_{0}V_{3} + Z \overline{Z}) b_{\overline{w}}}{(V_{3} + \overline{Z}\; b(w))^{2}} = 0$. 
\end{proof}

\vspace{0.4cm}

For the next result we use the definition that $$\Log z = \{w \in \complex : e^{w} = z \} = \ln \sqrt{z \overline z} + i(\arg(z) + 2 k \pi) 
\; \; \mbox{ where } \; \;  - \pi < \arg(z) < \pi.$$

\vspace{0.2cm}

\begin{theor}\label{1.11}
Let $(M,f)$ be a conformal immersion such that $f_w (w) = \mu(w) W(a(w), b(w))$. Then 
\begin{enumerate}
\item For $L_o$ and $L_3$ given in Proposition \ref{199} item (2), we have that
\begin{eqnarray}
 [f_{w w}]^{\perp} = \frac{\mu a_{w}}{1 - a \overline{b}} L_{0}(b) + \frac{\mu b_{w}}{1 - b \overline{a}} L_{3}(a).
 \end{eqnarray}  
$$f_{w \overline{w}} = \frac{\mu a_{\overline{w}}}{1 - a \overline{b}} L_{0}(b) + 
\frac{\mu b_{\overline{w}}}{1 - b \overline{a}} L_{3}(a).$$
Moreover the compatibility condition is given by $f_{w \overline{w}}(w) \in \lorentz^{4}$ for each $w \in M$.\\
\item  The compatibility condition implies that
\begin{equation}\label{1}
\frac{\mu a_{\overline{w}}}{1 - a \overline{b}} = 
\frac{{\overline{\mu}} \; {\overline{a}}_{w}}{1 - b \overline{a}} \; \; \mbox{ and } \; \; 
\frac{\mu b_{\overline{w}}}{1 - b \overline{a}} = 
\frac{{\overline{\mu}} \; {\overline{b}}_{w}}{1 - a \overline{b}},
\end{equation} 
hence 
\begin{equation}\label{1.1}
\frac{b_{\overline{w}} {\overline{a}}_{w}}{(1 - b \overline{a})^{2}} = 
\frac{a_{\overline{w}} {\overline{b}}_{w}}{(1 - a \overline{b})^{2}}.
\end{equation}  
\item  \ The factor of integration $\mu(w)$ satisfies 
\begin{equation}\label{2}
\frac{\partial}{\partial \overline{w}} Log \mu = 
\frac{\overline{b} a_{\overline{w}}}{1 - a \overline{b}} + \frac{\overline{a} b_{\overline{w}}}{1 - b \overline{a}}.
\end{equation}
\end{enumerate}
\end{theor}

\begin{proof}
Since the immersion is conformal $\lpr{f_{w}}{f_{w}} = 0$, hence  $\lpr{f_{w \overline{w}}}{f_{w}} = 0$. Thus $f_{w \overline{w}}$ is
an orthogonal vector fields of the Minkowski space $\lorentz^{4}$ along of the immersion $(M,f)$. Then
\begin{equation}\label{3}
f_{w \overline{w}} = \mu_{\overline{w}} W(a,b) + \mu W_{\overline{w}}(a,b) = \alpha L_{0}(b) + \beta L_{3}(a).
\end{equation}

Since $\lpr{W}{W} = 0$ it follows that $\lpr{W_{\overline{w}}}{W} = 0$, and hence  
$\lpr{f_{w \overline{w}}}{\overline{W}} = 0$. Now since 
$$W(a,b) = \weis{a}{b} = a(1,b,-ib,1) + (b,1,i,-b),$$
$$W_{w} = a_{w}(1,b,-ib,1) + b_{w}(1,a,-ia,-1),$$ 
$$W_{\overline{w}} = a_{\overline{w}}(1,b,-ib,1) + b_{\overline{w}}(1,a,-ia,-1),$$
$$L_{3}(a) = \nlight{a} = \overline{a} (a,1,i,a) + (1,a,-ia,-1),$$
$$L_{0}(b) = \slight{b} = \overline{b} (b,1,i,-b) + (1,b,-ib,1),$$  
we have in particular that $\lpr{L_{3}(a)}{L_{0}(b)} = -2(1 - a \overline{b})(1 - b \overline{a})$ and 
$$\lpr{W_{\overline{w}}}{L_{3}(a)} = 
-2 a_{\overline{w}} (1 - b \overline{a}), \ \ \ \ \lpr{W_{\overline{w}}}{L_{0}(b)} = -2 b_{\overline{w}} (1 - a \overline{b}).$$ 

Hence it follows from equation (\ref{3}) that $\alpha \lpr{L_{3}(a)}{L_{0}(b)} = \lpr{\mu W_{\overline{w}}}{L_{3}(a)}$ and then $$\alpha = 
\frac{\mu a_{\overline{w}}}{1 - a \overline{b}}.$$
In similar way one gets all the expressions of item 1. 

2) \ This is trivial, remembering that $\overline{a_{\overline{w}}} = \overline{a}_w$.

3) \ Since $\lpr{f_{w \overline{w}}}{\overline{W}} = 0$, it follows from Equation (\ref{3}) and from the expressions above, that 
$$\frac{\partial}{\partial \overline{w}} \Log \mu = - \frac{\lpr{W_{\overline{w}}}{\overline{W}}}{\lpr{W}{\overline{W}}} =
\frac{\overline{b} a_{\overline{w}}}{1 - a \overline{b}} + \frac{\overline{a} b_{\overline{w}}}{1 - b \overline{a}}.$$
\end{proof}

Our next result partially solves the question established in section above. More specifically,  it shows that the conditions which guarantee the existence of a conformal spacelike immersion 
$ (M,f)$ from $M$ into $\mathbb L^4$, such that $$f_w(w) = \mu W(a,b)=  \mu (a + b,1 + ab,i (1-ab),a- b),$$ for $a,b : M \to \mathbb C$ two smooth maps, are the existence of solutions of the PDEs system form by Equation (\ref{2}) together with the compatibility conditions (\ref{1}).

\begin{theor}\label{110}
Let $a,b$ maps from $M$ to $\complex$ such that $a(w) \overline{b(w)} \neq 1$ for each $w \in M$. 
If $\mu$ is a solution for the equation (\ref{2}) and the compatibility conditions (\ref{1}) hold, then 
the map 
$$f(w) = p_{0} + 2 \Re \int_{w_{0}}^{w} \mu(\xi) W(a(\xi),b(\xi)) d \xi,$$
is a conformal immersion $(M,f)$ from $M$ in $\mathbb L^4$ such that $f_{w}(w) = \mu(w) W(a(w),b(w))$.
\end{theor}

\begin{proof} 
Set  
$$\beta = \Phi dw = \mu W dw = (\mu (a + b) dw, \mu(1 + ab) dw, i\mu(1 - ab)dw, \mu(a - b)dw).$$ 
Since $\mu$ is a solution of the equation (\ref{2}) we have that 
$$\Phi_{\overline{w}} = \mu \left(- \frac{\lpr{\overline{W}}{W_{\overline{w}}}}{\lpr{W}{\overline{W}}} W + W_{\overline{w}}\right) = 
\frac{\mu a_{\overline{w}}}{1 - a \overline{b}}L_{0}(b) + \frac{\mu b_{\overline{w}}}{1 - b \overline{a}}L_{3}(a),$$  
which, by equations (\ref{1}), is a 
real valued vector for each $w \in M$.  It follows that $(M,f)$ is an immersion from $M$ in $\mathbb L^4$, for which, by Proposition \ref{119},  $f_{w}(w) = \mu(w)W(a(w),b(w))$, and so it also is conformal. In fact, 
one has that $f_{w\overline{w}}(w) = \Phi_{\overline{w}}(w)$ is a real valued vector of $\mathbb L^4,$ orthogonal to the spacelike plane $[\Re(W(a(w),b(w))),\Im(W(a(w),b(w)))] \subset \mathbb L^4$, \ with  $\left<f_w, f_w \right>=0$ and $\left<f_{w}, \bar{f_w} \right> = \frac{r^2}{2}$ for $r(w) >0$.   
\end{proof}

\begin{prop}\label{23}
Let $a$ and $b$ be two anti-holomorphic maps from $M$ into $\mathbb C$. For each holomorphic map
$h$ from $M$ into $\mathbb C$ we have that
$$\mu (w) = \frac{h(w)}{|1 - \overline {a(w)} b(w)|^2}$$
is a solution for the equation (\ref{2}). If $a$ and $b$ satisfy  the equation of compatibility (\ref{1}), then $h \ne  0$ is a real number.
\end{prop}

\begin{theor}\label{16}
If $\mu_{0}$ and $\mu_{1}$ are two solution of the equation (\ref{2}), then there exists a holomorphic map $h(w)$ from $M$ 
into $\complex$, such that $\mu_{1}(w) = h(w) \mu_{0}(w)$ for all $w \in M$. 

Moreover,  one has the following two cases.
\begin{enumerate}
\item  If $a,b \in \mathcal{H}(M)$ then $\mu \in \mathcal{H}(M)$, and the compatibility conditions (\ref{1}) hold trivially.

\item \ If $a$ or $b$ is not a holomorphic map and the compatibility conditions (\ref{1}) hold, then 
$\frac{\mu_{1}}{\mu_{0}} = h$ is a constant real valued map. 
\end{enumerate}
\end{theor}
\begin{proof}
Since $\mu_{0}$ and $\mu_{1}$ are two solutions of  (\ref{2})  it follows that $(\Log \mu_{1} - \Log \mu_{0})_{\overline{w}} = 0$.
Hence we have the holomorphicity $\frac{\mu_{1}}{\mu_{0}} = h \in \mathcal{H}(M)$. 

 (1) \ It is trivial.

(2) \ Assume without loss of generality that the map $a$ is not holomorphic, which means $a_{\overline{w}} \neq 0$. Then using equation
(\ref{1}) we have that  both maps 
 $\frac{\mu_{0} a_{\overline{w}}}{1 - a \overline{b}}$ and 
$\frac{h \mu_{0} a_{\overline{w}}}{1 - a \overline{b}}$ are real valued, which implies that $h(w)$ is a holomorphic real 
valued map, and so  $h_w = 0$. Hence $h$ is a constant (nonzero) real number.      
\end{proof}

\begin{corol}
Let $(M,f)$ and $(M,g)$ be two conformal immersions given respectively by 
$$f(w) = p_{0} + 2 \Re \int_{w_{0}}^{w} \mu_{0}(\xi) W(a(\xi),b(\xi)) d \xi \; \; \mbox{ and } \; \; 
g(w) = p_{1} + 2 \Re  \int_{w_{1}}^{w} \mu_{1}(\xi) W(a(\xi),b(\xi)) d \xi.$$ 
Then if $a(w)$ or $b(w)$ is not a holomorphic map,  there exists a real constant $k \neq 0$ such that 
\ $g(w) = (p_{1} -  k \; p_{0}) + k f(w)$ \ for each $w \in M$. 
\vspace{0.1cm}

Moreover, if $ds^{2}(f) = ds^{2}(g)$ and $k>0$,  then $g(w) = (p_{1} - p_{0}) + f(w)$, and so these two surfaces are 
congruent.
\end{corol}

\section{Conformal Immersions into the Future Directed Light Cone $\C$}

The future directed light cone $\C$ is the subset of the Minkowski space: 
$$\C = \{ L \in \lorentz^{4} | \lpr{L}{L} = 0, \ \;  L_{0} = - \lpr{L}{\partial_{0}} > 0 \}.$$
We note that if $A \in \C$ then there exist $\lambda \in \reals$ and $x \in \complex$, such that one of following expressions is obtained

$$ A = \lambda \nlight{x} \; \; \mbox{ or } \ \ A = \lambda (1,0,0,1), $$
$$ A = \lambda \slight{x} \; \; \mbox{ or } \; \; A = \lambda (1,0,0,-1).$$

\vspace{0.2cm}

\begin{lemma}\label{12}
For a connected open subset $M \subset \complex$, for a positive smooth function $\lambda: M \to \mathbb{R}$ and  
for a smooth function  $x: M \to \mathbb{C}$, let 
$$f(w) = \lambda (w) L(x(w)) = \lambda(w) \nlight{x(w)}$$ 
\vspace{0.1cm}
be a smooth map from $M$ in $\mink$. Then, $(M,f)$ is a conformal immersion from $M$ in the light cone 
$\C$ if, and only if, either 

\vspace{0.1cm}
{\rm (i)}  $x(w)$ is a holomorphic function from $M$  in $\mathbb{C}$ such that $\vert x_{w} \vert \neq 0$ for each $w \in M$, or 

\vspace{0.1cm}
{\rm (ii)}  $\overline{x}(w)$ is a holomorphic function (i.e \ $x(w)$ is anti-holomorphic function) from $M$ in 

\hspace{0.4cm} $\mathbb{C}$ such that $\vert \overline{x}_{w} \vert \neq 0$ for 
each $w \in M$.    
\end{lemma} 

\begin{proof}
First we note that $(M,f)$ is a conformal immersion if, and only if, 
\begin{equation}\label{291}
 x_{w} {\overline{x}}_{w} = 0  \; \; \; \; \; \mbox{ and } \; \; \; \;  \; \vert x_{w} \vert + \vert \overline{x}_{w} \vert \neq 0.
 \end{equation}
 Indeed, $f_{w} = \lambda_{w} L(x(w)) + \lambda x_{w}(\overline{x},1,-i,\overline{x}) + 
\lambda {\overline{x}}_{w}(x,1,i,x)$ thus, $\lpr{f_{w}}{f_{w}} = 0$ if, and only if, 
$$\lambda^{2} x_{w} {\overline{x}}_{w} \lpr{(\overline{x},1,-i,\overline{x})}{(x,1,i,x)} = 2\lambda^{2} x_{w} {\overline{x}}_{w} = 0.$$

Now, $\overline{f_{w}} = f_{\overline{w}} = \lambda_{\overline{w}} L(x(w)) + 
\lambda x_{\overline{w}}(\overline{x},1,-i,\overline{x}) + \lambda {\overline{x}}_{\overline{w}}(x,1,i,x)$. Thus, 
$\lpr{f_{w}}{f_{\overline{w}}} > 0$ if, and only if,  
$$\lpr{f_{w}(w)}{\overline{f_{w}(w)}} = 2 \lambda^{2} (\vert x_{w}(w) \vert^{2} + \vert \overline{x}_{w}(w) \vert^{2}) > 0.$$ 
Hence we have the assured. 
 
 \vspace{0.1cm}
 
Moreover, equations (\ref{291}) hold if, and only if, the sets $M_{1}$ and $M_{2} = M \setminus M_{1}$ defined by  
$$M_{1} = \{p \in M : \overline{x}_{w}(p) = 0 \neq \vert x_{w}(p) \vert\} \; \; \; \mbox{ and } \;  \; \; 
M_{2} = \{p \in M : x_{w}(p) = 0 \neq \vert \overline{x}_{w}(p) \vert\},$$ 
are closed in $M$. Since, $x_{w}$, $\overline{x}_{w}$ are continuous functions, if, 
$$p \in (\overline{M_{1}} \cap M_{2}) \cup (\overline{M_{2}} \cap M_{1})$$ then, 
we obtain $x_{w}(p) = 0 = \overline{x}_{w}(p)$ that contradicts second equation of  (\ref{291}). Thus, equations (\ref{291}) hold in all $M$ if, and only if, 
$\overline{M_{1}} \cap \overline{M_{2}} = \emptyset$, what means that the connected open subset 
$M \subset \complex$ is a disjoint union of these two open and closed subsets $M_{1}$ and $M_{2}$. Then, item (i) holds if, and only if, 
$M_{1} = M$ and $M_{2} = \emptyset$, and item (ii) holds if, and only if, $M_{2} = M$ and $M_{1} = \emptyset$.     
\end{proof}

\begin{lemma}\label{14}
Let $(M,f)$ be a conformal immersion from $M$ into the light cone $\C$ where $f(w) = \lambda(w) L(x(w))$, with  $x(w) \neq 0$ and $\lambda_{w}(w) \neq 0$ for all $w \in M$. Hence 
\begin{enumerate}
\item \ If $x(w)$ is a holomorphic map from $M$ into 
$\complex$, then  
$$\mu(f) = \overline{x} \lambda_{w}, \; \; b(f) = \frac{1}{\overline{x}}, \; \; 
a(f) = x + \frac{\lambda x_{w}}{\lambda_{w}}, \; \mbox{ thus } \; 
f_{w} = \overline{x} \lambda_{w} \; W\left( x + \frac{\lambda x_{w}}{\lambda_{w}}, \frac{1}{\overline{x}}\right).$$ 

\item \ If $x(w)$ is an anti-holomorphic map from $M$ into $\complex$, then  
$$\mu(f) = (\overline{x} \lambda)_{w}, \; \; a(f) = x, \; \; 
b(f) = \frac{\lambda_{w}}{(\lambda \overline{x})_{w}}, \; \mbox{ thus } \; 
f_{w} = (\overline{x} \lambda)_{w} \; W\left( x, \frac{\lambda_{w}}{(\lambda \overline{x})_{w}}\right).$$
\end{enumerate}
\end{lemma}
\begin{proof}
To show these facts we apply the formulas (\ref{1'}). Then we have
$${f_0}_{w} + {f_3}_{w} = 2 (\lambda x \overline{x})_{w}, \; \; \; \;  {f_0}_{w} - {f_3}_{w} = 2 \lambda_{w}, \; \; \; \; 
{f_1}_{w} - i {f_2}_{w} = 2 (\lambda \overline{x})_{w}.$$

So, if one assumes, for instance, that the map $x(w)$ is an anti-holomorphic map, 
it follows  
$$
a(f) = x, \; \; \; \; 
b(f) = \frac{f_{w}^{0} - f_{w}^{3}}{f_{w}^{1} - i f_{w}^{2}} =
\frac{\lambda_{w}}{\overline{x} \lambda_{w} + \overline{x}_{w} \lambda},
$$
and we have the expression for $f_w$ of item (2).
\end{proof}

We note that, the same conclusion of Lemma \ref{12} can be obtained  when it is assumed that  $f(w) = \rho (w) L_0 (x(w))$,  where by definition 
$$L_0 (x) = (1 + x \bar x, x + \bar x, -i(x - \bar x), 1 - x \bar x) \; \; \mbox{ for all } \; \; x \in \complex \; \mbox{ and } \; \rho  > 0.$$
In similar form we have a version of Lemma \ref{14} for this case, namely, 

\vspace{0.3cm}
(1) \   If $x(w)$ is a holomorphic map from $M$ into 
$\complex$, then  
\begin{equation} \label{15.1}
\mu(f) = \overline{x} \rho_{w}, \; \; a(f) = \frac{1}{\overline{x}}, \; \; 
b(f) = x + \frac{\rho x_{w}}{\rho_{w}}, \; \mbox{ thus } \; 
f_{w} = \overline{x} \rho_{w} \; W\left( \frac{1}{\overline{x}},   x + \frac{\rho x_{w}}{\rho_{w}} \right).
\end{equation}

\ (2) \ If $x(w)$ is an anti-holomorphic map from $M$ into $\complex$, then  
\begin{equation} \label{15}
\mu(f) = (\overline{x} \rho)_{w}, \; \; b(f) = x, \; \; 
a(f) = \frac{\rho_{w}}{(\rho \overline{x})_{w}}, \; \mbox{ thus } \; 
f_{w} = (\overline{x} \rho)_{w} \; W\left(\frac{\rho_{w}}{(\rho \overline{x})_{w}}, x \right).
\end{equation}

\begin{dfn}
A conformal immersion $(M,f)$ is totally umbilical if, and only if, $[f_{w w}(w)]^{\perp} = 0$ for all $w \in M$. 
Equivalently, for each unitary and orthogonal vector fields $\vec{n} \in [f_{w}]^{\perp}$, the shape 
operator $A_{f}^{(\vec{n})}$ is a multiple of the identity. \end{dfn}

Hence it follows, using Theorem 1.7, that the immersion  $(M,f)$ is totally umbilical if and only if $a_w = 0 = b_w$, that is 
$a, b \in {\overline{\mathcal{H}}}(M)$.

\vspace{0.3cm}

\begin{lemma}\label{19}
Let $a$, $b$ be two anti-holomorphic maps from $M$ into $\complex$ 
such that equation (\ref{1.1}) is satisfied. Then 
\begin{enumerate}
\item  It is possible to construct a conformal immersion $(M,f)$ from $M$ into the light cone $\C$, $f(w) = \lambda(w)L(a(w))$, 
such that $f_{w} = \mu(f) W(a,b)$ if, and only if, the function $\lambda(w)$ is real valued and satisfies the equation 
\begin{equation}\label{17}
\frac{\partial}{\partial w} \ln \lambda = \frac{\overline{a}_{w} b}{1 - \overline{a} b}.
\end{equation}

\item It is possible to construct a second conformal immersion $(M,g)$ from $M$ into the light cone $\C$, $g(w) = \rho(w) L_{0}(b(w))$,  
such that $g_{w} = \mu(g) W(a,b)$  if, and only if,the map $\rho(w)$ is real valued  and satisfies the equation
\begin{equation}\label{17.1}
\frac{\partial}{\partial w} \ln \rho = \frac{a\overline{b}_{w}}{1 - a\overline{b}}.
\end{equation}

\item For each $r \in \mathbb{R}$, $G(r,w) = p_{1} + f(w) + r g(w)$ is a conformal totally umbilical immersion in $\mathbb L^4$ such that 
$$G_{w} = (\mu(f) + r \mu(g))W(a,b).$$ 

\item  Given a totally umbilical immersion $(M,F)$ such that $F_w = \mu(F) W (a,b)$ there exist $r(F)  \in \mathbb R$ and a point $p_F \in
\mathbb L^4$ such that $G(r(F), w )= p_F + F(w)$.
\end{enumerate}
\end{lemma}

\begin{proof}

(1) Take  the conformal immersions $(M,f)$ from $M$ into the light cone $\C$ defined by formula $f(w) = \lambda(w) L(a(w))$. Then since $a$ is an 
anti-holomorphic map, we have, by (2) in Lemma \ref{14} that $a(f) = a$,  and moreover that $b(f) = b$  if and only if $b = \frac{\lambda_{w}}{(\lambda\overline{a})_{w}}$. But the latter happens  if, and only if,
$$ 
\frac{\partial}{\partial w} \ln \lambda = \frac{\overline{a}_{w} b}{1 - \overline{a} b}.$$ 
Now from (\ref{1.1}) it follows that $(\ln \lambda)_{w \bar w}$ is a real valued function which implies that one can choose the function $\lambda(w)$ being real valued.

(2) It is analogous to item (2) using $L_o(b)$ in formula (\ref{15}) for this case. 

(3) It is clear.

(4) From (2) of Theorem \ref{16} it follows that: $\mu(f) = t \mu(F)$ and $\mu(g) = s \mu(F)$ for $t, s \in \mathbb R$  with $t \neq 0 \neq s$. Therefore, for $r(F) = \frac{1 - t}{s}$ we have that 
$F_{w} = f_{w} + r(F) g_{w}$.
So,  $p_{1} = 
p_{0} - f(w_{0}) - r(F) g(w_{0})$, which all implies that $F(w) = G(r(F),w)$.
\end{proof}

Note that, using the above conditions, $\lpr{f + rg}{f + rg} = - 4r \rho \lambda \vert 1 - a \overline{b} \vert^{2} = 0$ if, and only if $r = 0$

\begin{prop}\label{14.1}
Let $a,b$ be two anti-holomorphic maps from $M$ into $\mathbb{C}$. Then 
\begin{enumerate}
\item  Equation (\ref{17})  has real valued solution if and only if  $(\Log(1 - b \overline{a}))_{w \overline{w}} \in C^{\infty}(M,\mathbb{R}).$  This happens if only if condition  (\ref{1.1}) is satisfied.

\item If  $a$, $b$ and $\lambda$ satisfy the statement (1) of Lemma \ref{19}, 
then a real valued solution of equation (\ref{17}) is given by:
\begin{equation}\label{21.1}
\lambda = \frac{h}{\overline{a}_{w}(1 - a \overline{b})}, \; \; \mbox{ for some } \; \; 0 < h \in \mathbb{R}.
\end{equation}  
\end{enumerate}
\end{prop}

\begin{proof}
(1) \  Since $a$ and $b$ are anti-holomorphic maps, we note that $$(\Log(1 - b \overline{a}))_{w \overline{w}} = - \frac{b_{\bar w} \bar a_w}{(1- b \bar a)^2}.$$ 
Hence it follows immediately that $(\Log(1 - b \overline{a}))_{w \overline{w}}\in C^{\infty}(M,\mathbb{R})$ if and only if  (\ref{1.1}) is satisfied. 

Now if $\lambda$ is solution of equation (\ref{17}) one has that
 $$(\Log(1 - b \overline{a}))_{w \overline{w}} = - (\ln \lambda)_{w \bar w} = - \frac{\lambda_{w \bar w} \lambda - \lambda_w \lambda_{\bar w}}{\lambda^2}.$$
Hence it follows that $\lambda$ is real valued if and only if $(\Log(1 - b \overline{a}))_{w \overline{w}}  \in C^{\infty}(M,\mathbb{R})$.
 
\vspace{0.3cm}

(2)\ By item 2 of Lemma \ref{14} , one has that $x = a$, \ $b(f) = b$, \ $\mu(f) = \overline{a}_{w} \lambda + \overline{a} \lambda_{w}$ with $\lambda_{w} = b \mu(f)$. Now from Proposition \ref{23}  we take the particular 
solution of Equation (\ref{2}) given by 
 $\mu = h/ \vert 1 - \overline{a}b\vert^{2}$ for someone positive $h \in \mathbb{R}$. Then from 
 $$
 \overline{a}_{w} \lambda + \overline{a} \lambda_{w} = h/ \vert 1 - \overline{a}b\vert^{2}
 $$
we get the particular solution $\lambda$ given by (\ref{21.1}).
\end{proof}


\begin{prop}\label{21}
Let $\vec v = (v_0, v_1, v_2, v_3)$ be an unitary vector in $\mathbb L^4$ and 
\newline  $V_0 = v_0 + v_3, \ V_3 = -v_0 + v_3$, \ $Z = v_1 + i v_2$. Then fixing an anti-holomorphic map $b \in \overline{\mathcal{H}}(M)$, the equation 
$$\frac{\partial}{\partial w} \ln \lambda = \frac{b{\overline{a}}_{w}}{1 - b \overline{a}} \; \; \mbox{ where } \; \; 
a(w) = \frac{V_{0} b(w) - Z}{V_{3} + \overline{Z} b(w)},$$
has a real valued solution given by 
$$\lambda = k \left( \frac{V_{3}^{2} + V_{3}(Z \overline{b} + \overline{Z} b) + Z \overline{Z} b \overline{b}}
{V_{3} + Z \overline{b} + \overline{Z} b - V_{0} b \overline{b}} \right), \ \  with \ \   k > 0.$$
\end{prop}
\begin{proof} Since $b$ is anti-holomorphic map the given map $a(w)$ is also anti-holomorphic. Moreover,  
$$\frac{\partial}{\partial w} \ln \lambda = - \frac{\partial}{\partial w} \Log (1 - b \overline{a})
\; \; \mbox{ where } \; \; \lambda \in C^{\infty}(M,\reals),$$ 

$$- \Log(1 - b \overline{a}) = 
- \Log \left( \frac{V_{3} + Z \overline{b} + \overline{Z} b - V_{0} b \overline{b}}{V_{3} + Z \overline{b}} \right),$$ 
and   $V_{3} + Z \overline{b} + \overline{Z} b - V_{0} b \overline{b} \in \mathbb R$. Hence 
it follows that 
$$\frac{\partial}{\partial w} \ln(\lambda (V_{3} + Z \overline{b} + \overline{Z} b - V_{0} b \overline{b})) = 
\frac{\partial}{\partial w} \Log( V_{3} + Z \overline{b}).$$
Now since $\bar b$ and $Log$ are holomorphic maps  we have that $\frac{\partial}{\partial w} \Log( V_{3} + Z \overline{b})$ is also holomorphic. 
Therefore, $\ln(\lambda (V_{3} + Z \overline{b} + \overline{Z} b - V_{0} b \overline{b})) = \ln k +  \ln((V_{3} + Z \overline{b})(V_{3} + b \overline{Z}))$ for each $k > 0$.
\end{proof}

\begin{theor}\label{29}
Let $a,b \in \overline{\mathcal{H}}(M)$. If, the equations (\ref{17}) and (\ref{17.1}), namely, 
$$\frac{\partial}{\partial w} \ln \lambda = \frac{b{\overline{a}}_{w}}{1 - b \overline{a}} \; \; \mbox{ and } \; \;
\frac{\partial}{\partial w} \ln \rho = \frac{a{\overline{b}}_{w}}{1 - a \overline{b}}$$ 
can be solved for $\lambda$ and $\rho$ real valued functions, then there is a M$\ddot{o}$bius transformation $M_{\vec{v}}$ 
such that $a(w) = M_{\vec{v}}(b(w))$ for each $w \in M$. 
\end{theor}
\begin{proof}
We define $f(w) = \lambda(w) L(a(w))$ and $g(w) = \rho(w) L_{0}(b(w))$. Then by Lemma \ref{19}, 
we have two conformal immersions from $M$ into the light cone $\C$, such that
$$f_{w} = \eta W(a,b) \; \; \mbox{ and } \; \; g_{w} = \xi W(a,b).$$ 

We are assured it is possible to take $\lambda$ and $\rho$ such that $\lambda \rho (1 - a \overline{b})(1 - b \overline{a})$ is  a positive constant real valued function. 
Indeed, if we have two solutions of equation (\ref{17}), namely, $\lambda_{1}$ and $\lambda_{2}$  (similarly for equation (\ref{17.1})), then, 
there is a nonzero real number $t \in \reals$ such that $\lambda_{2} = t \lambda_{1}$, so we may pick $\lambda$ and
$\rho$  
such that $f$ and $g$ are two future directed light like vector fields of $\lorentz^{4}$. Now, from
$$\frac{\partial}{\partial w} \ln(\lambda \rho (1 - a \overline{b})(1 - b \overline{a})) = 
\frac{\partial}{\partial w} \ln \lambda + \frac{\partial}{\partial w} \ln \rho - 
\frac{b{\overline{a}}_{w}}{1 - b \overline{a}} - \frac{a{\overline{b}}_{w}}{1 - a \overline{b}} = 0,$$ 
then it follows that $\lambda \rho (1 - a \overline{b})(1 - b \overline{a})$ is constant function a positive real valued, as we wanted. 

Next we denote that constant function $\lambda \rho (1 - a \overline{b})(1 - b \overline{a})$ by $r$.
Then for that $r$ we define the conformal immersion in $\mathbb L^4$ 
$$h(w) = f(w) + g(w)$$  with conditions 
$$ \lpr{h(w)}{h(w)} = -4r \; \ \ \mbox{ and } \; \ \  h_{w} = \mu W(a,b).$$ 
Now from Theorem \ref{16}, it follows that $\eta(w) = k \xi(w)$ where $k \in \reals$ is a nonzero real constant.
Then for that $k$ we define a vector field orthogonal to $[h_{w}]$, namely,  
$$n(w) = f(w) - k g(w).$$
One notes that 
$$ n_{w} = (\eta - k \xi) W(a,b) = 0, \; \ \ \ \ \forall w \in M.$$
Therefore it follows that $n(w)$ is a constant vector field orthogonal to the surface $S = h(M)$.  From Proposition \ref{0.1} it follows that for \ $\vec v = \frac{n(w)}{| n(w)|}$ \ the M\"obius  transformation $M_{\vec{v}}$ \ is such that 
$a(w) = M_{\vec{v}}(b(w))$. Finally one also notes  that,   
if $k > 0$ (resp. $k < 0$) then $\vec{v}$ is a spacelike vector (resp. timelike vector) in $\lorentz^{4}$. 
\end{proof}

\begin{theor}\label{1101}
Let $(M,F)$ be a conformal totally umbilical immersion from $M$ into $\mathbb L^4$ such that $a(F)$ and $b(F)$ are non constants 
anti-holomorphic maps. Then 
\begin{enumerate}
\item  There exists a conformal immersion $(M,f)$ from $M$ into the light cone $\C$, such that 
$$a(F) = a(f), \; \; \; b(F) = b(f), \ \ \mu(F) =  \mu(f).$$

\item There exists an affine hyperplane $H(w_{0},\vec{v}) = F(w_{0}) + [\vec{v}]^{\perp}$ such that $F(M) \subset H(w_{0},\vec{v})$. Moreover, 
\end{enumerate}

{\rm (2.1)} If $\lpr{\vec{v}}{\vec{v}} = -1$ then $F(M)$ is congruent to a spherical surface $S^{2}(1/r)$ of the

\hspace{0.8cm} Euclidean space $\mathbb R^{3}$. 

\vspace{0.2cm}

{\rm (2.2)} If $\lpr{\vec{v}}{\vec{v}} = 1$ then $F(M)$ is congruent to a hyperbolic surface $S^{2}(-1/r)$ of the 

\hspace{0.8cm} Lorentzian space $\mathbb R_{1}^{3}$. 
\end{theor}

\begin{proof}
(1) From hypothesis we have that $a(F)$ and $b(F)$ satisfy equations (\ref{1}) and (\ref{1.1}).  Now it takes the real solution of Equation (\ref{17}) for $a(F)$ and $b(F)$, given by formula (\ref{21.1}) for $h = 1$, namely, 
$$\lambda = \frac{1}{\overline{a(F)}_{w} (1 - a(F)\overline{b(F)})}.$$ 
 Then  we construct the conformal immersion $g$ in the light cone $\mathcal C$ given by $g(w) = \lambda(w) L(a(F)(w))$ with 
$g_w = \mu(g)  W(a(F), b(F))$, for which from Lemma \ref{14}, one also knows that $\mu(g) = \overline{a(F)}_w \lambda + \overline{ a(F)} \lambda_w$.  Now since $\mu(F)$ and $\mu(g)$ are solutions of Equation (\ref{2}), it follows from Theorem \ref{16} item 2, that there exists a constant real number $c$ such that $\frac{\mu(F)}{\mu(g)} = c$.

 Next we observe that we can choose implicitly $c=1$, since we can construct a new  conformal immersion $f$ from $M$ into the light cone $\C$, such that 
$a(F) = a(f), \  b(F) = b(f)$ and $\mu(F) =  \mu(f).$ In fact we take $f$ such that $\mu(f) = c \mu(g)$, so $\mu(f) = \overline{a(F)}_w (c \lambda) + \overline{ a(F)} (c \lambda)_w$, which implies that we are taking the solution $c\lambda$ of Equation (\ref{17}). Then $f$ has the expression $f(w) = c\lambda(w) W(a(F), b(F))$.

\vspace{0.2cm}
  
  (2)  It follows from Lemma \ref{19} item 4, that there exists $t \in \mathbb R$ such that $$F(w) = f(w) + t g(w),$$ for $f$ and $g$ conformal immersions in the light cone $\mathcal C$, given by $f(w) = \lambda(w) L(a(F)(w))$,
 $g(w) = \rho(w) L_0 (b(F) (w))$, where $\lambda$, $\rho$ are real valued functions, with $f_w = \mu(f) W(a(F),b(F))$, \ $g_w = \mu(g) W(a(F),b(F))$. Moreover, one has that equations (\ref{17}) and (\ref{17.1}) are satisfied.   
 Then we use the same argument used in the proof of Theorem \ref{29},  to get the vector field 
 $$
n(w) = f(w) - tk g(w)$$
where $k \in \mathbb R - \{0\}$  is such that $\mu(f) = t k \mu(g)$. Then it follows that $\vec{v} = \frac{n(w)}{\vert n(w)\vert}$
is a unit constant vector field orthogonal to the surface $F(w)$. Hence by  Proposition \ref{0.1}  we have that $F(M) \subset H(w_{0},\vec{v})$, where  $H(w_{0},\vec{v}) = F(w_{0}) + [\vec{v}]^{\perp}$ is 
the affine hyperplane associated to $\vec v$.

 Moreover,  if $\left< \vec v,  \vec v \right> = - 1 $ then $F(M)$ is contained in a spacelike affine hyperplane. Since $F$ is totally umbilical in $\mathbb L^4$, it follows that  $F(M)$ is congruent to a spherical surface $S^{2}(1/r)$ of the Euclidean space $\mathbb R^{3}$. But if $\left< \vec v,  \vec v \right> =1 $ it follows that $F(M)$ is contained in a timelike affine hyperplane and hence  $F(M)$ is congruent to an hyperbolic surface $S^{2}(-1/r)$ of the Lorentzian space $\mathbb R_{1}^{3}$.  

  \end{proof}

\section{When the Mean Curvature Vector $H(f)$ is a Lightlike Vector Field}

This section is dedicated to studying equations which occur when we assume the geometric condition of mean curvature vector $H(f)$ to be  lightlike vector field and its relation to conformal immersions into the 
hyperbolic space $\mathbb H^3$. Here we will finish solving the initial question about under which conditions there exists conformal immersion $(M,f)$  in $\mathbb L^4$ such that $f_w = \mu W(a,b)$, now assuming $a\in \mathcal H(M)$ but $b$ is neither holomorphic or anti-holomorphic function.

\vspace{0.2cm}
We start by noting that  the mean curvature vector $H(f)$ is given by $H(f) = \frac{2}{g_{11}(f)} f_{w \overline{w}}$.  Then
 from Theorem \ref{1.11} it follows that $H(f)$  is a 
lightlike vector field in $\lorentz^{4}$, if, and only if, either $a(w)$ is a holomorphic map 
and $b(w)$ is not a holomorphic map, or,  $a(w)$ is not a holomorphic map and $b(w)$ is holomorphic. 
Moreover we see that $H(f) = 0$, if, and only if $f_{w \overline{w}} = 0$, and this last equation holds if, and only if, 
$a, b \in {\mathcal{H}}(M)$. 

\vspace{0.3cm}

Our fist result is consequence of  Lemma \ref{14}.

\begin{prop}\label{32}
Let $(M,f)$ be a conformal immersion from $M$ into the light cone $\C$, such that $f(w) = \lambda(w) L(x(w))$ with 
$x \in {\mathcal{H}}(M)$.  Then the mean curvature vector $H(f)$ is a lightlike vector field of $\mathbb L^4$, if and only if, $a(f) \in {\mathcal{H}}(M)$.

 Moreover  on that condition, we have that
 
 \begin{enumerate}
 \item  \ $\frac{\partial}{\partial w} \ln \lambda = \frac{x_{w}}{a - x} \; \; \mbox{ therefore } \; \; 
\lambda(w) = \lambda_{0} e^{2 \Re \int_{w_{0}}^{w} \frac{x_{\xi}}{a(\xi) - x(\xi)} d \xi}.$
\item  \ $f_{w \overline{w}} = \lambda_{w \overline{w}} L \left(x + \frac{\lambda x_{w}}{\lambda_{w}}\right) \; \; \mbox{ and } \; \;
(\ln \lambda)_{w \overline{w}} = \frac{\lambda \lambda_{w \overline{w}} - \lambda_{w} \lambda_{\overline{w}}}{\lambda^{2}} = 0.$

\item \ $g_{11}(f) = 4 \lambda^{2} \vert x_{w} \vert^{2} \; \; \mbox{ thus } \; \; 
H(f) = \frac{\lambda_{w \overline{w}}}{ 2 \lambda^{2} \vert x_{w} \vert^{2}} 
L\left(x + \frac{\lambda x_{w}}{\lambda_{w}}\right)$ \\

\item \ $ K(f) = 0,$  where $K(f)$ is the Gaussian curvature of surface $S = f(M)$. 
\end{enumerate}
\end{prop}

\begin{proof}
From  Lemma \ref{14} item (1), one has that $b(f) = \frac{1}{\bar x}$ is anti-holomorphic map, and so $H(f)$ is a lightlike vector field of $\mathbb L^4$, if and only if, $a(f) \in {\mathcal{H}}(M)$.
\vspace{0.1cm}

Now, since that $a(f) \in {\mathcal{H}}(M)$, again  from Lemma \ref{14} item (1),  it follows that $a(f) =  x + \frac{\lambda x_w}{\lambda_w}$, which implies immediately  that $\frac{\partial}{\partial w} \ln \lambda = \frac{x_{w}}{a - x} $.  By integrating one has the expression for $\lambda$ and so we have proved item (1).

\vspace{0.2cm}

For (2), we that
$$
f_{w \overline w} = \frac{\mu(f) b_{\overline w}}{1 - \overline a b} L(a)=  \frac{\lambda_w}{\lambda} \lambda_{\overline w} L(x + \frac{\lambda x_w}{\lambda _w}) = \frac{\lambda_{\overline w} x_w}{a - x} L(x + \frac{\lambda x_w}{\lambda _w}) = \lambda_{w \bar w} L(x + \frac{\lambda x_w}{\lambda _w}).$$
The second statement is obvious.

\vspace{0.2cm}
(3) \  Since $ds^2(f) = 4 \mu \overline \mu |1 - a(f) \bar b(f)| dw d\overline w$, \  $\mu = \bar x \lambda_w$ and $\lambda$ is real valued function, we have that $ds^2(f) = 4 \lambda^2 |x_w|^2 $. So, $g_{11}(f) = 4 \lambda^2 |x_w|^2$. The other statement is immediately.

\vspace{0.2cm}

(4) \ Since $K(f) = - \frac{\Delta \ln (g_{11}(f))}{g_{11}(f)}$, $x \in \mathcal H(M)$ and $(\ln \lambda)_{w \bar w}=0$, we get that $K(f) =0$.
\end{proof}

\begin{prop}\label{392}
Let $(M,f)$ and $(M,g)$  two conformal immersions  from $M$ into the light cone $\C$, given by
$$f(w) = \lambda(w) L(x(w)) \; \; \mbox{ and } \; \; g(w) = \rho(w) L(y(w))$$
where $x,  y  \in \mathcal{H}(M)$. \   If the immersions satisfy that 
\begin{equation}\label{30}
a(f) = a(g) \in \mathcal{H}(M) \ \ \ {and} \ \ \  \lambda \rho (x - y)(\overline{x} - \overline{y}) = 1,
\end{equation}
then $h(w) = \frac{1}{2} (f(w) + g(w))$ defines a conformal immersion from $M$ into the hyperbolic space $\hyper$
with lightlike mean curvature vector.
\end{prop}

\begin{proof}
In fact using Equation (\ref{30}) one has that $\left<L(x) , L(y) \right>= - \frac{2}{\lambda \rho}$, which implies that
$\left<h(w), h(w)\right> = -1$.
Now using that $a(f) = a(g)$ and the expressions $$f_w = \bar x \lambda_w W(a(f), b(f)), \; \; \; \; \ g_w = \bar y \rho_w W(a(g), b(g))$$ we have that
$\left<h_w, h_w\right> = \frac{1}{2} \left<f_w , g_w \right> = 0.$ 

Now, since $\left<h_w, h\right> =0$ we have that $\left<h_w, h_{\bar w}\right> = - \left<h_{w \bar w}, h \right> \ne 0$.  Since $h_{w \bar w}$ and $h$ are, respectively, lightlike and timelike vectors, it follows that $ \left<h_{w \bar w}, h \right> < 0$, and hence $\left<h_w, h_{\bar w}\right>>0$. 

Moreover, using the expression for $f_{w \bar w}$ and  $g_{w \bar w}$ given in Proposition \ref{32}, we have that  
$$
\left<H(h), H(h)\right> = \frac{1}{g_{11}^2(h)} (\lambda^2_{w \bar w} + 2 \lambda_{w \bar w} \rho_{w \bar w} + \rho^2_{w \bar w}) \left<L(a), L(a)\right> = 0.
$$
\end{proof}

\subsection{The Linking Equation for Surfaces in $\mathbb H^3$}

 In this subsection we establish a interesting partial equation relating the maps $a(w),  b(w)$ of a conformal immersion from $M$ into the hyperbolic space. We will see in next section that after a change variable that equation will become into the partial Riccati's equation with holomorphic parameters. 

\vspace{0.2cm}

Let $h : M \to \lorentz^{4}$ be a conformal immersion from $M$ into $\hyper$, thus $\lpr{h}{h} = -1$.
From these conditions it follows that
$$ h_{w} = \mu W(a,b) \; \; \mbox{ where } \; \; W(a,b) = \weis{a}{b}$$
$$L_{3}(a) = \nlight{a}$$
$$L_{0}(b) = \slight{b}.$$

Since $\lpr{h}{h} = -1$, one has that $\lpr{h_{w}}{h} = 0$ which implies that $$h(w) = \alpha(w) L_{3}(a(w)) + \beta(w) L_{0}(b(w)).$$
Since  
$\lpr{h_{w}}{L_{0}} = 0 = \lpr{h_{w}}{L_{3}} $ it follows  that $ \frac{\alpha_{w}}{\alpha} = - \frac{\lpr{{L_{3}}_{w}}{L_{0}}}{\lpr{L_{3}}{L_{0}}}$ \ and \ 
$\frac{\beta_{w}}{\beta} = - \frac{\lpr{{L_{0}}_{w}}{L_{3}}}{\lpr{L_{3}}{L_{0}}}.$ Now after some computations we get that 
\begin{equation}
\frac{\alpha_{w}}{\alpha} = \frac{a_{w} \overline{b}}{1 - a \overline{b}} + 
\frac{{\overline{a}}_{w} b}{1 - \overline{a} b}, \   \; \; \; \; \ \; \; \; 
\frac{\beta_{w}}{\beta} = \frac{b_{w} \overline{a}}{1 - b \overline{a}} + 
\frac{{\overline{b}}_{w} a}{1 - \overline{b} a} 
\end{equation}
Since $\alpha$ and $\beta$ are real valued smooth maps, we have that 
$\overline{\frac{\alpha_{w}}{\alpha}} = \frac{\alpha_{\overline{w}}}{\alpha}$ and 
$\overline{\frac{\beta_{w}}{\beta}} = \frac{\beta_{\overline{w}}}{\beta}$, hence 
\begin{equation}\label{35}
\frac{\alpha_{\overline{w}}}{\alpha} = \frac{{\overline{a}}_{\overline{w}} b}{1 - b \overline{a}} + 
\frac{{\overline{b} {a}}_{\overline{w}}}{1 - a \overline{b}}, \ \; \; \; \ \; \; \;
\frac{\beta_{\overline{w}}}{\beta} = \frac{{\overline{b}}_{\overline{w}} a}{1 - a \overline{b}} + 
\frac{{b}_{\overline{w}} \overline{a}}{1 - \overline{a} b}.
\end{equation}
Now, since $\left<h_w, h\right> = 0$ and $W(a, b)$ is a lightlike vector, it follows that
$ \lpr{h_{ww}}{h} = - \lpr{h_{w}}{h_{w}} = 0$, and hence we get that  
\begin{equation}\label{34} 
\alpha a_{w} (1 - b \overline{a}) = - \beta b_{w} 
(1 - a \overline{b}).
\end{equation}

\vspace{0.2cm}

It notes that, from Equation (\ref{34}),  it follows that $a \in \overline{\mathcal H}(M)$ if, and only if, 
$b  \in \overline{\mathcal H}(M)$. So, under such condition we would have a totally umbilical surface $(M,h)$.

\vspace{0.2cm}
Hence next we assume that $a_{w} \neq 0 \neq b_{w}$. Then, taking the logarithmic derivative relative to $\overline{w}$ on Equation (\ref{34}) we have:
\begin{equation}\label{40}
\frac{\alpha_{\overline{w}}}{\alpha} + \frac{a_{w{\overline{w}}}}{a_{w}} - 
\frac{b {\overline{a}}_{\overline{w}}}{1 - b \overline{a}} - 
\frac{\overline{a} b_{\overline{w}}}{1 - b \overline{a}} = 
\frac{\beta_{\overline{w}}}{\beta} + \frac{b_{w{\overline{w}}}}{b_{w}} - 
\frac{a {\overline{b}}_{\overline{w}}}{1 - a \overline{b}} - 
\frac{\overline{b} a_{\overline{w}}}{1 - a \overline{b}}.
\end{equation}

\vspace{0.3cm}
\begin{theor} Let $h : M \to \lorentz^{4}$ be a conformal immersion from $M$ into $\hyper$ with $h_w = \mu W(a, b)$,  $a_{w} \neq 0 \neq b_{w}$ and $a\ne b$. Then
\begin{enumerate}
\item The functions $a$ and $b$ are linked by the equation   
\begin{equation}\label{36.1}
\frac{a_{w{\overline{w}}}}{a_{w}} + \frac{2 \overline{b}}{1 - a \overline{b}} a_{\overline{w}} = 
\frac{b_{w{\overline{w}}}}{b_{w}} + \frac{2 \overline{a}}{1 - b \overline{a}} b_{\overline{w}}.
\end{equation}

\vspace{0.1cm}

\item  If $a \in \mathcal{H}(M)$ then the linking equation is 
\begin{equation}\label{39.1}
b_{w{\overline{w}}} + \frac{2 \overline{a}}{1 - b \overline{a}} b_{w} b_{\overline{w}} = 0 \; \; 
\mbox{ with the condition } \; \; b_{w} \; b_{\overline{w}} \neq 0.
\end{equation}

\vspace{0.1cm}

\item  Taking $\varphi = \frac{1}{\overline{b}}$, Equation (\ref{39.1}) becomes the partial Riccati's equation with holomorphic parameters, namely, 
\begin{equation}\label{111}
\frac{\varphi_{w}}{(\varphi - a)^{2}} = P, \; \; \mbox{ where } \; \; a,P \in \mathcal{H}(M).
\end{equation}
\end{enumerate}
Moreover, the smooth map $\varphi$ from $M$ into $\complex$ is given by the conformal 
immersion $h:M \to \hyper$, and $\varphi$ is neither a holomorphic map nor an anti-holomorphic map.  
\end{theor}

\begin{proof}
(1) \ Equation (\ref{36.1}) follows from substituting, in Equation (\ref{40}), the values obtained in Equation (\ref{35}).  

(2)  It is immediate.

(3)  Explicitly the map $P$ is given by $P= -\frac{{\overline b}_w}{(1 - a \overline b)^2}$, which by the hypotheses it is immediately holomorphic. The other statements are obvious.
\end{proof}

Next we construct examples of conformal surfaces $h(M) \subset \mathbb H^3$ which satisfy the linked equation  (\ref{36.1}).

\begin{example} Let $F : B_{1}(0) \to \hyper$ be the parametrization of the hyperbolic space $\hyper$ given by 
$$F(x,y,z) = \frac{1}{1 - x^{2} - y^{2} - z^{2}}(1 + x^{2} + y^{2} + z^{2}, 2x, 2y, 2z),$$ 
where   
$$B_{1}(0) =\{(x,y,z) \in \mathbb R^{3} : x^{2} + y^{2} + z^{2} < 1\}.$$
We have that 
$$\lpr{F_{x}}{F_{x}} = \lpr{F_{y}}{F_{y}} = \lpr{F_{z}}{F_{z}} = \frac{4}{(1 - x^{2} - y^{2} - z^{2})^{2}},$$ 
and $\lpr{F_{x}}{F_{y}} = \lpr{F_{x}}{F_{z}} = \lpr{F_{y}}{F_{z}} = 0$. \ Therefore, for each conformal immersion 
$$\beta(w) = (x(w),y(w),z(w)) \in B_{1}(0) \; \; \mbox{ for } \; \; w \in M,$$ 
the pair $(M,F \circ \beta)$ is an immersion from $M$ in $\hyper$. 

If $(M,\beta)$ is a part of a non flat minimal immersion contained in the ball $B_{1}(0)$, then $a(F \circ \beta)$ and 
$b(F \circ \beta)$ satisfy Equation (\ref{36.1}), and they are not anti-holomorphic maps. 

\end{example}

\section{Geometric solutions associated to the Partial Riccati's Equation with Holomorphic Parameters}

In this section we focus for obtaining geometric solutions of Equation (\ref{111}), that means solutions $\varphi$ which are related to conformal spacelike surfaces in
$\mathbb H^3$ with lightlike mean curvature vector.

We start with the following lemma which gives us technique for construing solutions of Equation (\ref{111}).

 \begin{lemma}\label{54}
Let $x$ and $y$ be two distinct holomorphic solutions of the Riccati equation 
$$\frac{\varphi_{w}}{(a - \varphi)^{2}} = P, \ \ \ {with} \ \ \ a, P \in \mathcal H(M).$$ 
Then, taking 
$$\lambda = \lambda_{0} e^{2 \Re \int_{w_{0}}^{w} P(\xi)(a(\xi) - x(\xi)) d \xi} \; \; \mbox{ and } \; \; 
\rho = \rho_{0} e^{2 \Re \int_{w_{0}}^{w} P(\xi)(a(\xi) - y(\xi)) d \xi},$$ 
it follows that equations (\ref{30}) are satisfied.
\end{lemma}
\begin{proof}
Taking the immersions $f(w) = \lambda(w) L(x(w))$ and $g(w) = \rho(w) L(y(w))$ and computing the functions $a(f)$ and $a(g)$ 
we have that 
$$a(f) = x + \frac{\lambda x_{w}}{\lambda_{w}} = x + \frac{P(a - x)^{2}}{P(a - x)} = x + (a - x) = a. \; \; 
$$
Analogously  $a(g) = a$, and so, we have proved first equation of  (\ref{30}). 
 
Now we need to show that $\lambda \rho (x - y)(\overline{x} - \overline{y}) = r$ for some real number $r > 0$. In fact, taking logarithm derivative and using that $x$ and $y$ are solutions of Riccati equation, we have that 
$$\frac{r_{w}}{r} = \frac{\lambda_{w}}{\lambda} + \frac{\rho_{w}}{\rho} + \frac{x_{w} - y_{w}}{x - y} = 
2aP - P(y + x) + P\frac{(x - a)^{2} -(y - a)^{2}}{x - y} = 0.$$ 
\end{proof}
\vspace{0.3cm}

Next we find explicit parametrization of spacelike surfaces which are associated to solutions of Equation  (\ref{111}).

\begin{example}\label{46.2}
It takes $a = 1/2 + \log w$ and $P = \frac{4}{w}$ in Equation  (\ref{111}). Then the functions $x = \log w$,  and, 
$y = 1 + \log w$ where $ w \in \complex \setminus \{ u + i v : u \leq 0, \ v = 0 \}$, are explicit solutions of (\ref{111}). According to Lemma \ref{54}, we can take, for instance,  $\lambda = (w{\overline{w}})^{2}$,  $\rho = 1 / \lambda$, and we conclude that equations (\ref{30}) are satisfied. Hence we find that the conformal spacelike surface in $\mathbb H^3$ with lightlike mean curvature vector, which is the associated solution to Equation  (\ref{111}), is given by $h(w) = \frac{1}{2}(f(w) + g(w))$ where $f(w) = \lambda
L(x)$ and $g(w) = \rho L(y)$. Then  $h(w)$ has explicitly the following components
$$
h(w) = \frac{1}{2 |w|^4} (2 + 2 \Re(log (w)) + |log(w)|^2(|w|^8 + 1) + |w|^8,  2 \Re(log (w)) (|w|^8 + 1)+2, 
$$
$$ 2 \Im (log (w)) (|w|^8 + 1), 2 \Re(log (w)) + |log(w)|^2 (|w|^8 + 1) - |w|^8 ),
$$ 
which corresponds to $$\varphi = \frac{(|w|^8 -1) log(w) -1}{|w|^8 -1}.$$
\end{example}

\begin{example}\label{46.3}
It takes $a = \frac{2 w}{3}$ and $P = \frac{3}{w^2}$ in Equation  (\ref{111}). Then the functions $x = \frac{w}{3}$  and 
$y = \frac{4w}{3}$ are explicit solutions of (\ref{111}). According to Lemma \ref{54}, we can take, for instance,  $\lambda = w{\overline{w}}$,  $\rho = \frac{1}{(w \overline w)^2}$, and we conclude that equations (\ref{30}) are satisfied. Hence we find that the conformal spacelike surface in $\mathbb H^3$ with lightlike mean curvature vector, which is the associated solution to Equation  (\ref{111}), is given by $h(w) = \frac{1}{2}(f(w) + g(w))$ where $f(w) = \lambda L(x)$ and $g(w) = \rho L(y)$. Then  $h(w)$ has explicitly the following components
$$
h(w) = \frac{1}{2 |w|^4} (1 + \frac{16}{9} |w|^2 + |w|^6 + \frac{|w|^8}{9},  \frac{2}{3} \Re(w) (|w|^6 + 4),  \frac{2}{3} \Im (w) (|w|^6 + 4),  
$$
$$- 1 + \frac{16}{9} |w|^2 - |w|^6 + \frac{|w|^8}{9} ),
$$ 
which corresponds to $$\varphi = \frac{w (|w|^6 - 8)}{3 (|w|^6 - 2)}.$$
\end{example}
\vspace{0.3cm}

The following example is a surface which belongs to the family of the Catenoid Cousins. We reference the papers \cite{GG} and \cite{F}, where are studied the cousins surfaces and in particular the Catenoid cousins. Similarly we invite the reader to see the second example of \cite{B} on page 341 for details. 

\begin{example}\label{222}
We start taking 
$$x = - y = e^{w} \; \; \mbox{ and } \; \; \lambda = \frac{1}{2}e^{-u + v}, \; \; \rho = \frac{1}{2}e^{-u - v}.$$ 
Then since $2 \lambda = e^{-u + v}$ we have that $\ln 2 + \ln \lambda = -u + v$, and hence  $\frac{\lambda_{w}}{\lambda} = -(1 + i)/2.$
It follows that
$$a = x + \frac{\lambda x_{w}}{\lambda_{w}} = e^{w} - e^{w}\frac{2}{1 + i} = e^{w}\frac{-1 + i}{1 + i} = i e^{w},$$ 
$$P = \frac{x_{w}}{(a - x)^{2}} = \frac{e^{w}}{e^{2w}} \left(\frac{1 + i}{2}\right)^{2} = \frac{i}{2} e^{-w}.$$ 
Now we take $f(u,v) = \lambda(u,v)(1 + \vert x(u,v)\vert^{2}, 2 \Re(x), 2\Im(x), -1 + \vert x(u,v)\vert^{2})$, for 
$x = e^{u}(\cos v + i \sin v)$ and $\lambda = \frac{1}{2} e^{-u}e^{v}$, we obtain 
$$f(u,v) = e^{v}(\cosh u, \cos v, \sin v, \sinh u).$$ 
Analogous computation using $y$ and $\rho$, it gives us 
$$g(u,v) = e^{-v}(\cosh u, -\cos v, -\sin v, \sinh u).$$ 
Therefore we have that the conformal spacelike surface in $\mathbb H^3$ associated to solution of Equation  (\ref{111}), is given by 
$$h(u,v) = \frac{1}{2}(f(u,v) + g(u,v)) = (\cosh u \cosh v, \sinh v \cos v, \sinh v \sin v, \sinh u \cosh v),$$
which corresponds to $$\varphi = -i e^w \frac{\cosh v + i \sinh v}{\cosh v - i \sinh v}.$$

From a direct computation we can see that: 
$$ds^{2}(h) = \cosh^{2} v (du^{2} + dv^{2}) \; \; \mbox{ and } \; \; \lpr{h(u,v)}{h(u,v)} = -1.$$ 
Now, we know that the mean curvature vector $H(h)$ is a lightlike vector field. In fact, we write $h(u,v) = A(u,v) + B(u,v)$ where 
$$A(u,v) = \cosh v(\cosh u, 0, 0, \sinh u) \; \; \mbox{ and } \; \; B(u,v) = \sinh v (0, \cos v, \sin v, 0).$$ 
It follows that $A_{uu} + A_{vv} = 2 A$ and $B_{uu} + B_{vv} = 2 \cosh v(0, -\sin v, \cos v, 0)$, therefore 
$$\lpr{h_{uu} + h_{vv}}{h_{uu} + h_{vv}} = - 4 \cosh^{2} v + 4 \cosh^{2} v = 0.$$ 
The Gauss curvature of this surface is:  
$$K(u,v) = \frac{-1}{2g_{11}} \Delta \ln g_{11} = 
\frac{-1}{\cosh^{2} v} \frac{\partial}{\partial v}\left(\frac{\sinh v}{\cosh v}\right) = \frac{-1}{\cosh^{4} v}.$$

\vspace{0.1cm}
From these last facts we can see that the surface  $h(M)$ immersed in the
hyperbolic space is locally isometric to the minimal Helicoid (or
Catenoid) of the Euclidean space, which are, respectively,  parametrized by
$$H(u,v) = (0,\sinh v \cos u,\sinh v \sin u, u), \; \; \; \;
C(u,v) = (0,\cosh v \cos u,\cosh v \sin u, v).$$

\end{example}

\vspace{0.3cm}

Next we continue our study of Riccati equation. In fact, we will transform Riccati Equation (\ref{111}) into a  linear partial differential 
equation of second order, for which we will obtain general solutions. This process will give us a way to obtain local solutions of Equation (\ref{111}).

\vspace{0.1cm}
It takes $\phi = \varphi - a$, so Equation (\ref{111}) becomes into  
$$\phi_{w} = P \phi^{2} - a_{w}.$$
 Next, it takes  $\phi = \frac{- X_{w}}{X P}$, with
$X P \ne 0$ and $X$ being a not anti-holomorphic function, 
then we have the following linear partial differential 
equation of second order:
\begin{equation}\label{112}
P X_{ww} - P_{w} X_{w} - a_{w} P^{2} X = 0.
\end{equation}

Now we look for general solutions for Equation (\ref{112}), more generally, for a second order partial differential equation (PDE) 
\begin{equation}\label{46}
y_{ww} + a y_{w} + by = 0.
\end{equation}

We start with the second order ordinary differential equation (ODE) on complex domain, namely, 
$$y'' + a y' + by = 0, \ \ \ a,b \in \mathcal{H}(M),$$
where,  by definition  
$$y'(w) = \lim_{h \rightarrow 0} \frac{y(w + h) - y(w)}{h}.$$ 

Then it is well known that a general solution of ODE can be written as 
$$y(w) = c_{1} y_{1}(w) + c_{2} y_{2}(w) \; \; \mbox{ for } \; \; c_{1},c_{2} \in \complex,$$ 
where $y_{1}(w)$ and $y_{2}(w)$ are two linearly independent solution of the equation. Hence if we take two holomorphic solutions $\{y_{1},y_{2}\}$  of the ODE and assume that 
$c_{1},c_{2}$ are anti-holomorphic  maps, we have that a general solution of the associated 
PDE (\ref{46}),  is given by 
$$y(w) = c_{1}(w) y_{1}(w) + c_{2}(w) y_{2}(w), \; \; \mbox{with} \; \; c_{1},c_{2} \in \overline{\mathcal{H}(M)}.$$

\vspace{0.3cm}

We observe that every solution of ODE is, in fact, a solution of the associated PDE 
 (\ref{46}), but in other way, there exist solutions of Equation (\ref{46}) which are not solutions of ODE.

\vspace{0.3cm}
Before continuing the process of obtaining the solutions of (\ref{111}) , we note that Equation (\ref{112}) has already appeared  in other context when we assume particular values of $a$ and $P$, as follows.

\begin{example}
Let \ $a_{w} = - G(w)$ and $P(w) = 1$ for $w \in M$. Then Equation (\ref{112}) becomes  to $X_{ww} + G(w) X = 0$, well known as 
equation of the Complex Oscillatory Theory. 
\end{example}

\begin{example}
Let \ $a(w) = l w^{-k - 1}$ and $P(w) = w^{k}$ with $k, l$ constant complex numbers, then  Equation (\ref{112}) 
becomes to $X_{ww} - \frac{k}{w} X_{w} + \frac{l(k-1)}{w^{2}} X = 0$. An Euler's equation.
\end{example}
\begin{example}
Taking $a(w) = k w$ and $P(w) = 1$ for all $w \in M$, and for the constant complex number $k$, the equation (15) 
becomes $X_{ww} - k X = 0$. 
\end{example}

\begin{example}
From Example \ref{222}, we have that $a(w) = i e^w$ and $P(w)= \frac{i}{2} e^{-w}$, so Equation (\ref{112}) becomes to
$$
2X_{ww}  - 2X_w + X =0.$$
\end{example}

Then we retake the technique above to construct non anti-holomorphic solutions of Riccati Equation 
(\ref{111}) or of its equivalent equation $\phi_{w} = P \phi^{2} - a_{w},$  as follows.

\vspace{0.2cm}
In general form, let $P,Q$ \and $R$ be three holomorphic maps from $M$ into $\complex$ such that $P(w) \neq 0$ for each $w \in M$, and let 
\begin{equation}\label{47}
\frac{\partial \psi}{\partial w} = P \psi^{2} + Q \psi + R.
\end{equation}
the partial Riccati equation of holomorphic parameters.

\begin{prop}\label{46.1}
Let $x,y \in \mathcal{H}(M)$ two distinct holomorphic solution of the partial Riccati's equation 
of holomorphic parameters (\ref{47}). \ Then  a smooth map $\psi$ from $M$ into $\complex$ satisfy the partial Riccati's equation (\ref{47})  if, and only if,
\begin{equation}\label{48}
\frac{\psi - x}{\psi - y} = k \theta(w) e^{\int_{w_{0}}^{w} P(\xi)(x(\xi) -y(\xi)) d \xi} \; \mbox{ where } \; 
k \in \complex \; \; \ \mbox{and} \; \; \theta \in \overline{\mathcal{H}}(M).
\end{equation}
\end{prop}

\begin{proof}
Since $\psi$ and $x$ are solutions of (\ref{47}), it follows that $\psi_{w} - x_{w} = P(\psi - x)(\psi + x) + Q(\psi - x)$.  Hence since $y$ is also solution, we have that 
$$\frac{\psi_{w} - x_{w}}{\psi - x} - \frac{\psi_{w} - y_{w}}{\psi - y} = P(x - y)$$
and this last is equivalent to  
$$\frac{\partial}{\partial w} \left( \Log \left( \frac{\psi - x}{\psi - y} \right) - \int P(x - y) dw \right) =0.$$
Now, since $x,y,P$ are holomorphic maps from $M$ into $\complex$, it follows the direct statement.

For converse, we apply the logarithmic derivative to (\ref{48}) and replace the values of $x_{w}$ and  $y_{w}$ in the result, 
to obtain Equation (\ref{47}). 
\end{proof}

We finish this section showing that every conformal immersion $(M,F)$ of a spacelike surface in $\mathbb L^4$ in which the PDE (\ref{39.1}) is satisfied, it is congruent to a Bryant immersion $(M,h)$ in $\mathbb H^3$ whose integration factor $\mu$ and maps $a$ and $b$ are the same of $F$. For the proof we use Theorem \ref{16}, Lemma \ref{14}, Proposition  \ref{32}, Proposition \ref{392} and Lemma \ref{54}.

\begin{theor}\label{96} 
 Let $(M,F)$ be a conformal immersion such that the maps $a(F), b(F)$ satisfy the PDE equation (\ref{39.1}), with $a(F) \in \mathcal H(M)$. Then,
 
 (1) The mean curvature vector $H(F)$ is a lightlike vector fields along $S = F(M)$. 
 
 (2) There exists a Bryant immersion $(M, h)$ from $M$ in the hyperbolic space $\mathbb H^3$ such that
$$
\mu (h) = \mu (F), a(h) = a(F), \ \ \  and  \ \ \ b(h) = b(F).
$$
Moreover, $S$ and $h(M)$ are congruent each other by the translation vector $\vec v =
F(w_0) - h(w_0)$.
\end{theor}

\begin{proof}
(1) \ Since $a(F) \in \mathcal{H}(M)$ and $b \notin \mathcal{H}(M)$, it follows that the mean curvature vector $H(F)$ is a lightlike vector field.

\vspace{0.3cm}
(2) \  From Equation (\ref{39.1}) we have the PDE Riccati Equation 
$$P(w) = \frac{\varphi_{w}}{(a(F) - \varphi)^{2}} \; \; \mbox{ for } \; \; \varphi(w) = \frac{1}{\overline{b(F)(w)}}.$$ 
Now it takes two linearly independent holomorphic solutions $\{x,y\}$ of the Riccati Equation above. Then we find from Lemma \ref{54} that for 
$$\lambda = \lambda_{0} e^{2 \Re \int_{w_{0}}^{w} P(\xi)(a(\xi) - x(\xi)) d \xi} \; \; \mbox{ and } \; \; 
\rho = \rho_{0} e^{2 \Re \int_{w_{0}}^{w} P(\xi)(a(\xi) - y(\xi)) d \xi},$$ 
the conditions (\ref{30}) are satisfied. Hence we have that
$$h(w) = \frac{1}{2}(f(w) + g(w))$$
 with $f(w) = \lambda(w) L(x(w)), \ g(w) = \rho(w) L(y(w))$, defines a conformal immersion
from $M$ into the hyperbolic space $\mathbb H^3$, with lightlike mean curvature vector.   Initially, we note that $a(f) = a(g)$, and hence one gets that $a(F) = a(f) = a(g)$. From expression of
$a(h)$, it follows that $a(h) = a(f) = a(g)$ and hence we have $a(h) = a(F)$. Henceforward we denote it simply by $a$.  

\vspace{0.2cm}
 Next our goal is to choice conveniently $\lambda$ and $\rho$ such that  $\mu(h) = \mu(F)$ and $b(h) = b(F)$.
 
 \vspace{0.2cm}
If we denote $b(h) = \frac{1}{\overline{\phi}}$, then we already know that  
$$f_{w} = \eta W(a,1/\overline{x}), \; \; g_{w} = \xi W(a,1/ \overline{y}) \; \; \ \text {and} \ \ \ h_{w} = \mu(h) W(a,1/ \overline{\phi}),$$
hence  we will require that 
$$\mu(h) = \mu(F) \; \; \mbox{ and } \; \; \phi = \varphi.$$

 For this purpose, we go through to get explicit expressions of $\lambda$ and $\rho$ which satisfy the conditions.
 We force that  $F_w = h_w$, that means, 
$$2 \mu(F)  W(a, 1/{\overline{\varphi}}) = \eta W(a, 1/{\overline{x}}) + \xi W(a, 1/{\overline{y}}).$$ 
Then from this assumption we get the system 
\begin{equation}
\mu(F) = \frac{\eta + \xi}{2}  \; \;  \ \ \ \ \mbox{ and } \; \; \ \ \ \frac{\mu (F)}{\overline{\varphi}} = \frac{1}{2} \left(\frac{\eta}{\overline{x}} + 
\frac{\xi}{\overline{y}} \right),
\end{equation} 
which has solution given by
$$
 \frac{\xi}{\overline{y}} = 2 \frac{\mu(F)}{\overline{\varphi}} \;
\frac{\overline{\varphi} - \overline{x}}{\overline{y} - \overline{x}} \; \; \; \ \ \ 
\mbox{ and } \; \; \;  \ \ \ 
 \frac{\eta}{\overline{x}} = 2 \frac{\mu(F)}{\overline{\varphi}} \;
\frac{\overline{\varphi} - \overline{y}}{\overline{x} - \overline{y}}.
$$ 

Now  using that $\xi = \overline{y} \rho_{w}$ and 
$\eta = \overline{x} \lambda_{w}$, we get that 
$$ \rho_{w} =  2 \frac{\mu(F)}{\overline{\varphi}} \;
\frac{\overline{\varphi} - \overline{x}}{\overline{y} - \overline{x}} \ \ \ \ \ \text{and} \ \ \ \ \
\lambda_{w} =  2 \frac{\mu(F)}{\overline{\varphi}} \;
\frac{\overline{\varphi} - \overline{y}}{\overline{x} - \overline{y}},
$$
therefore, since $\lambda$ and $\rho$ are real valued functions it follows that
\begin{equation}
\frac{\rho_{\overline{w}}}{\lambda_{\overline{w}}} = - \left( \frac{\varphi - x}{\varphi - y} \right).
\end{equation}
Now, from Lemma \ref{14} item 1. and Proposition \ref{32}, one has that
$$\xi_{\overline{w}} = {\overline{a}} \frac{\rho_{w} \rho_{\overline{w}}}{\rho} 
\; \; \; \ \ \ \  \text{and} \; \; \; \ \ \ \eta_{\overline{w}} = 
{\overline{a}} \frac{{{\lambda_{w}}}{\lambda_{\overline{w}}}}{\lambda},$$
which implies that $\xi_{\overline{w}} / {\overline{a}}$ and $\eta_{\overline{w}} / {\overline{a}}$  
are positive real valued maps. Now we use the fact that 
$\frac{\rho_{w}}{\rho} = P(a - y)$ and $\frac{\lambda_{w}}{\lambda} = P(a - x)$, for getting 
$$\frac{\xi_{\overline{w}}}{\eta_{\overline{w}}} = \frac{\rho_{\overline{w}}}{\lambda_{\overline{w}}} 
\frac{\rho_{w}}{\rho} \frac{\lambda}{\lambda_{w}} = - \left( \frac{\varphi - x}{\varphi -y}\right)\left(\frac{y-a}{x-a}\right),$$
and therefore that  
$- \left(\frac{\varphi - x}{\varphi -y} \right)  \left(\frac{y-a}{x-a} \right)$ is positive real valued map. So, we have proved that  
\begin{equation}\label{122}
z = - \left( \frac{\varphi -x}{\varphi - y} \right)\left(\frac{\overline{x} - \overline{a}}{\overline{y} - \overline{a}} \right)
\end{equation}
is also positive real valued map. 

Finally then we take $\lambda(w)$ and $\rho(w)$ given by the formulas 
\begin{equation}\label{123}
\lambda(w) = \frac{1}{\vert x(w) - y(w) \vert \sqrt{z(w)}}, \ \ \ \ \ \ \ \rho(w) = \frac{\sqrt{z(w)}}{\vert x(w) - y(w) \vert}
\end{equation}
and we observe that these $\lambda$ and $\rho$ are wanted.

In fact, they satisfy the conditions 
$$\lambda \rho \vert x - y \vert^{2} = 1, \ \ \ \frac{\rho(w)}{\lambda(w)} = z(w),$$
$$\frac{\rho_{w}}{\rho} = P(a - y), \ \ \ \ \frac{\lambda_{w}}{\lambda} = P(a - x), \ \ \ \ \ x - y + (\frac{\lambda x_{w} }{ \lambda_{w}}) - (\frac{\rho y_{w}} { \rho_{w}}) = 0,$$

and so, in particular Formula (\ref{122}) and (\ref{123}) are consistent with Proposition \ref{392}.

\vspace{0.2cm}

Now from $h_{w} = (f_{w} +  g_{w}) / 2$, \ $f_{w} = \eta W(a,1/{\overline{x}})$ and $g_{w} = \xi W(a,1/{\overline{y}})$, it follows that
$$b(h) = \frac{1}{\overline{ \phi}} = \frac{\lambda_w + \rho_w}{\bar x \lambda_w + \bar y \rho_w}$$
and therefore, using Formula (25), we get that 
 
$$\phi = \frac{\lambda_{\overline{w}} x + \rho_{\overline{w}} y}{\lambda_{\overline{w}} + \rho_{\overline{w}}} = 
\frac{x + (\rho_{\overline{w}} / \lambda_{\overline{w}})y}{1 + (\rho_{\overline{w}} / \lambda_{\overline{w}})} = \varphi.$$ 

\vspace{0.3cm}

Since  $\eta = \bar x \lambda_w$ and $\xi = \bar y \rho_w$, it follows that 
$$\eta + \xi = \bar x \left( 2 \frac{\mu(F)}{\overline{\varphi}} \;
\frac{\overline{\varphi} - \overline{y}}{\overline{x} - \overline{y}}\right) + \bar y \left(2 \frac{\mu(F)}{\overline{\varphi}} \;
\frac{\overline{\varphi} - \overline{x}}{\overline{y} - \overline{x}}\right) = 2 \mu(F),$$
which implies that $\mu(h) = \mu(F)$.

\vspace{0.3cm}

Finally, it follows from Theorem \ref{16} that $S$ and $h(M)$ are congruent each other by the translation $\vec{v} = F(w_{0}) - h(w_{0})$ 
since 
$$F(w) = \left(h(w_{0}) + 2\Re \int_{w_{0}}^{w} \mu(h)(\xi) W(a(h)(\xi),b(h)(\xi)) d \xi\right) + \vec{v}.$$  
\end{proof}

\vspace{0.3cm}

{\bf Acknowledgments}

\vspace{0.3cm}
The first author's research was supported by Projeto Tem\'atico Fapesp n. 2016/23746-6. S\~ao Paulo. Brazil.

\end{document}